\documentclass{amsart}
\usepackage{verbatim}
\usepackage[textsize=scriptsize]{todonotes}
\usepackage{longtable}

\usepackage{etoolbox}
\newtoggle{final}
\toggletrue{final}

\usepackage{tikz}
\usetikzlibrary{matrix,arrows}
\usepackage{tikz-cd}
\usepackage{stmaryrd}
\usepackage{listings}

\usepackage{fancyhdr}

\usepackage[margin=1in]{geometry}
\usepackage{amsmath}
\usepackage{amssymb}
\usepackage{amsthm}
\usepackage{amscd}
\usepackage{enumerate}
\usepackage[pdfusetitle,unicode,hidelinks]{hyperref}
\usepackage{bbm}
\usepackage{etoolbox}

\usepackage[utf8]{inputenc}
\usepackage[T1]{fontenc}

\newtheorem{proposition}{Proposition}
\newtheorem{corollary}[proposition]{Corollary}
\newtheorem{lemma}[proposition]{Lemma}
\newtheorem{theorem}[proposition]{Theorem}

\newtheorem*{conjecture*}{Conjecture}
\newtheorem*{theorem*}{Theorem}
\newtheorem*{corollary*}{Corollary}
\newtheorem*{proposition*}{Proposition}
\newtheorem*{lemma*}{Lemma}
\theoremstyle{definition}
\newtheorem{definition}[proposition]{Definition}

\newtheorem*{definition*}{Definition}
\newtheorem*{construction*}{Construction}
\theoremstyle{remark}
\newtheorem{remark}[proposition]{Remark}
\newtheorem*{remark*}{Remark}

\newtheorem{example}[proposition]{Example}
\newtheorem*{example*}{Example}

\newcommand{\id}{\operatorname{id}}
\newcommand{\Z}{\mathbb{Z}}
\def\C{\mathbb C}
\newcommand{\N}{\mathbb{N}}
\newcommand{\Q}{\mathbb{Q}}
\newcommand{\F}{\mathbb{F}}

\let\scr=\mathcal
\let\bb=\mathbb
\newcommand{\Gm}{{\mathbb{G}_m}}
\newcommand{\Gmp}[1]{{\mathbb{G}_m^{\wedge #1}}}

\def\P{\bb P}
\def\R{\bb R}
\newcommand{\1}{\mathbbm{1}}

\newcommand{\eff}{{\text{eff}}}
\newcommand{\veff}{{\text{veff}}}

\newcommand{\SH}{\mathcal{SH}}

\DeclareMathOperator*{\colim}{colim}

\let\lim=\relax
\DeclareMathOperator*{\lim}{lim}
\def\Map{\mathrm{Map}}

\def\map{\mathrm{map}}
\def\CAlg{\mathrm{CAlg}}

\def\Cat{\mathcal{C}\mathrm{at}{}}

\newcommand{\Spec}{\mathrm{Spec}}

\newcommand{\wequi}{\simeq}

\DeclareRobustCommand{\ul}{\underline}
\newcommand{\heart}{\heartsuit}

\newcommand{\Hom}{\operatorname{Hom}}

\def\op{\mathrm{op}}

\let\cat=\mathrm

\def\Sch{\cat{S}\mathrm{ch}{}}

\def\Nis{\mathrm{Nis}}
\def\Zar{\mathrm{Zar}}

\newcommand{\et}{{\acute{e}t}}
\newcommand{\ret}{{r\acute{e}t}}

\newcommand{\lra}[1]{\langle #1 \rangle}
\def\ph{\mathord-}

\numberwithin{proposition}{section}
\def\comp{\wedge}
\def\HZ{\mathrm{H}\mathbb{Z}}
\def\MWZ{\mathrm{H}\widetilde{\mathbb{Z}}}
\def\MGL{\mathrm{MGL}}
\def\KGL{\mathrm{KGL}}
\newcommand{\kgl}{\mathrm{kgl}}
\def\KO{\mathrm{KO}}
\newcommand{\ko}{\mathrm{ko}}
\def\KW{\mathrm{KW}}
\newcommand{\kw}{\mathrm{kw}}
\def\GW{\mathrm{GW}}
\def\Pro{\mathrm{Pro}}
\newcommand{\cell}{\text{cell}}
\newcommand{\Premot}{\mathcal{C\!M}}

\newcommand{\scomp}{\mathrm{sc}}
\def\Schl{\Sch_{\Z[1/\ell]}}
\def\PrL{\mathcal{P}\mathrm{r}^\mathrm{L}}
\newcommand{\vcd}{\mathrm{vcd}}
\newcommand{\cof}{\mathrm{cof}}
\newcommand{\fib}{\mathrm{fib}}
\newcommand{\Sper}{\mathrm{Sper}}

\iftoggle{final} {
\renewcommand{\todo}[1]{}
\newcommand{\NB}[1]{}
\newcommand{\tom}[1]{}
\newcommand{\paul}[1]{}
}{ % else
\newcommand{\TODO}[1]{\todo[color=red]{#1}}
\newcommand{\NB}[1]{\todo[color=gray!40]{#1}}
\newcommand{\tom}[1]{\todo[color=green!40]{#1}}
\newcommand{\paul}[1]{\todo[color=blue!40]{#1}}
}

\title{Topological models for stable motivic invariants of regular number rings}

\date{\today}

\author{Tom Bachmann}
\address{Department of Mathematics, LMU Munich, Germany}
\email{tom.bachmann@zoho.com}

\author{Paul Arne {\O}stv{\ae}r}
\address{Department of Mathematics, University of Oslo, Norway}
\email{paularneostvar@gmail.com}

\begin{document}

\maketitle

\begin{abstract}
For an infinity of number rings we express stable motivic invariants in terms of topological data determined 
by the complex numbers, 
the real numbers, 
and finite fields.
We use this to extend Morel's identification of the endomorphism ring of the motivic sphere with the 
Grothendieck-Witt ring of quadratic forms to deeper base schemes.
\end{abstract}

\tableofcontents

\section{Introduction}

The mathematical framework for motivic homotopy theory has been established over the last twenty-five years 
\cite{Levine-app}.
An interesting aspect witnessed by the complex and real numbers,
$\C$, $\R$,
is that Betti realization functors provide mutual beneficial connections between the motivic theory and the corresponding 
classical and $C_{2}$-equivariant stable homotopy theories \cite{levine2014comparison}, 
\cite{GWXspecialfiber},
\cite{heller2018stable}, 
\cite{BehrensShah2019},
\cite{ElmantoShah2019},
\cite{isaksen-stablestems}, 
\cite{io-survey}.
We amplify this philosophy by extending it to deeper base schemes of arithmetic interest.
This allows us to understand the fabric of the cellular part of the stable motivic homotopy category of $\Z[1/2]$ 
in terms of $\C$, $\R$, and $\F_3$ --- the field with three elements.
If $\ell$ is a regular prime,
a number theoretic notion introduced by Kummer in 1850 to prove certain cases of Fermat's Last Theorem \cite{Was-cyclotomic}, 
we show an analogous result for the ring $\Z[1/\ell]$.

For context, 
recall that a scheme $X$, 
e.g., an affine scheme $\Spec(A)$, 
has an associated pro-space $X_\et$, 
denoted by $A_\et$ in the affine case, 
called the \emph{étale homotopy type} of $X$ representing the étale cohomology of $X$ with coefficients in local systems; 
see \cite{artin2006etale} and \cite{friedlander1982etale} for original accounts and \cite[\S5]{hoyois2018higher} for a 
modern definition. 
For specific schemes,
$X_\et$ admits an explicit description after some further localization, 
see the work of Dwyer--Friedlander in \cite{dwyer1983conjectural,dwyer1994topological}.
For example, they established the pushout square
\begin{equation}
\label{equation:etalehomotopytype}
\begin{CD}
\C_\et^\comp @>>> \R_\et^\comp \\
@VVV                 @VVV      \\
(\F_3)_\et^\comp @>>> \Z[1/2]_\et^\comp
\end{CD}
\end{equation}
Here the completion $(\ph)^\comp$ takes into account the cohomology of the local coefficient systems $\Z/2^n(m)$.

\begin{remark}
If $k$ is a field, 
then $k_\et$ is a pro-space of type $K(\pi,1)$, 
where $\pi$ is the Galois group over $k$ of the separable closure of $k$. 
If $S$ is a henselian local ring with residue class field $k$, 
then $k_\et\rightarrow S_\et$ is an equivalence (by the affine analog of proper base change \cite{gabber1994affine}).
For instance, 
$\C_\et\wequi *$ is contractible,  
$\R_\et \wequi \R\P^\infty$ is equivalent to the classifying space of the group $C_2$ of order two, 
and $(\F_p)_\et \wequi (\Z_p)_\et$ is equivalent to the profinite completion of a circle.
That is, 
up to completion, 
\eqref{equation:etalehomotopytype} can be expressed more suggestively as $\Z[1/2]_\et \wequi S^1 \vee \R\P^\infty$.
For our generalization to stable motivic homotopy invariants, 
it will be essential to keep track of the fields and not just their étale homotopy types.
\end{remark}

The presentation of $\Z[1/2]_\et^\comp$ has powerful consequences;
for example, 
taking the $2$-adic étale $K$-theory of \eqref{equation:etalehomotopytype} yields a pullback square.
Combined with the Quillen--Lichtenbaum conjecture for the two-primary algebraic $K$-theory of $\Z[1/2]$, 
see \cite{mb}, \cite{weibel-cras}, \cite{oestvaer-QL}, \cite{hodgkin2003homotopy}, 
one obtains the pullback square 
\begin{equation}
\label{equation:Ktheorydiagram}
\begin{CD}
K(\Z[1/2])_2^\comp @>>> K(\R)_2^\comp \\
@VVV                  @VVV    \\
K(\F_3)_2^\comp @>>> K(\C)_2^\comp
\end{CD}
\end{equation}

We show that replacing algebraic $K$-theory in \eqref{equation:Ktheorydiagram} by an arbitrary 
\emph{cellular motivic spectrum} over $\Z[1/2]$ still yields a pullback square.
Let $\SH(X)$ denote the motivic stable homotopy category of $X$,  
see \cite{Jardine-spt}, \cite{mot-functors}, \cite[\S5]{morel-trieste}, \cite[\S4.1]{bachmann-norms}.
We write $\SH(X)^\cell \subset \SH(X)$ for the full subcategory of cellular motivic spectra \cite{mot-cell-str}, 
i.e., 
the localizing subcategory generated by the bigraded spheres $S^{p,q}$ for all integers $p,q\in \Z$.
For simplicity we state a special case of Theorem \ref{thm:main-1}, see Example \ref{ex:Z12}.

\begin{theorem}
\label{theorem:mainthmintro}
For every $\mathcal{E} \in \SH(\Z[1/2])^\cell$ there is a pullback square
\begin{equation}
\label{equation:Ediagram}
\begin{CD}
\mathcal{E}(\Z[1/2])_2^\comp @>>> \mathcal{E}(\R)_2^\comp \\
@VVV                  @VVV    \\
\mathcal{E}(\F_3)_2^\comp @>>> \mathcal{E}(\C)_2^\comp
\end{CD}
\end{equation}
Here, 
for $X \in \Sch_{\Z[1/2]}$, 
we denote by $\mathcal{E}(X)$ the (ordinary) spectrum of maps from $\1_X$ to $p^*\mathcal{E}$ in $\SH(X)$, 
where $\1_X \in \SH(X)$ denotes the unit object and $p: X \to \Z[1/2]$ is the structure map.
\end{theorem}

\begin{example}
The motivic spectra 
%$\KGL, \KO, \KW, \HZ, \MGL$ 
representing algebraic $K$-theory, $\KGL$, 
hermitian $K$-theory, $\KO$, Witt-theory, $\KW$, motivic cohomology or higher Chow groups, $\HZ$, 
and algebraic cobordism, $\MGL$, are cellular (at least after localization at $2$) 
by respectively \cite[Theorem 6.2]{mot-cell-str}, \cite[Theorem 1]{rondigs2016cellularity}, 
\cite[Theorem 1]{rondigs2016cellularity}, \cite[Proposition 8.1]{hoyois-algebraic-cobordism} and 
\cite[Corollary 10.4]{spitzweck2012commutative},  \cite[Theorem 6.4]{mot-cell-str}.
We refer to \cite[Proposition 8.12]{bachmann-eta} for cellularity of the corresponding 
(very effective or connective) covers $\kgl$, $\ko$, $\kw$,
in the sense of \cite{spitzweck2012motivic}, 
and Milnor-Witt motivic cohomology $\MWZ$, 
in the sense of \cite{bcdfo2020milnorwitt}, \cite{bachmann-etaZ}.

In the case of $\mathcal{E}=\KGL$, 
Theorem \ref{theorem:mainthmintro} recovers the stable version of \cite[Theorem 1.1]{hodgkin2003homotopy}, 
and for $\mathcal{E}=\KO$ it recovers \cite[Theorem 1.1]{berrick2011hermitian} 
(in fact, we extend these results to arbitrary $2$-regular number fields, not necessarily totally real).
The squares for $\KW$, $\HZ$, $\MWZ$, $\MGL$, $\kgl$, $\ko$, $\kw$ appear to be new.
\end{example}

A striking application of Theorem \ref{theorem:mainthmintro} is that it relates the universal motivic invariants 
over $\Z[1/2]$ to the same invariants over $\C$, $\R$, and $\F_{3}$.
That is, 
applying \eqref{equation:Ediagram} to the motivic sphere $\mathcal{E}=\1_{\Z[1/2]}$ enables computations of the  
stable motivic homotopy groups of $\Z[1/2]$.
We identify,
up to odd-primary torsion, 
the endomorphism ring of $\1_{\Z[1/2]}$ with the Grothendieck-Witt ring of quadratic forms of the 
Dedekind domain $\Z[1/2]$ defined in \cite[Chapter IV, \S3]{milnor1973symmetric}.
This extends Morel's fundamental computation of $\pi_{0,0}(\1)$ over fields \cite[\S6]{morel-trieste} to an 
arithmetic situation.

\begin{theorem}
\label{theorem:2mainthmintro}
The unit map $\1_{\Z[1/2]} \to \KO_{\Z[1/2]}$ induces an isomorphism 
\[
\pi_{0,0}(\1_{\Z[1/2]}) \otimes \Z_{(2)} 
\cong 
\GW(\Z[1/2]) \otimes \Z_{(2)}
\]
\end{theorem}

\begin{remark}
The étale homotopy types of various other rings and applications to algebraic $K$-theory and group homology of 
general linear groups were worked out in \cite{dwyer1983conjectural}, \cite{dwyer1994topological}, 
\cite{oestvaer-swissalps}, \cite{hodgkin2003homotopy}.
We show similar generalizations of \eqref{equation:Ediagram} with $\Z[1/2]$ replaced by $\scr O_F[1/2]$, 
for $F$ any $2$-regular number field, 
or by $\Z[1/\ell]$, $\Z[1/\ell,\zeta_\ell]$, 
where $\ell$ is an odd regular prime and $\zeta_\ell$ is a primitive $\ell$-th root of unity; 
to achieve this we slightly alter the other terms in \eqref{equation:Ediagram}.
See Theorems \ref{thm:main-1}, \ref{thm:main-2}, \ref{thm:main-3}, \ref{theorem:general2mainthmintro} 
for precise statements.
\end{remark}

Another application, 
which will be explored elsewhere, 
is the spherical Quillen--Lichtenbaum property saying the canonical map from stable motivic homotopy groups to 
stable étale motivic homotopy groups is an isomorphism in certain degrees.
Slice completeness is an essential input for showing the spherical property; 
we deduce this for base schemes such as $\Z[1/2]$ in Proposition \ref{prop:slice-app}.

As a final comment, 
we expect that most of the applications we establish hold over more general base schemes, 
where convenient reductions to small fields are not possible.
The proofs will require significantly different ideas.

\subsection*{Organization}
In \S\ref{sec:nilpotent-completion} we give proofs for some more or less standard facts about nilpotent completions 
in stable $\infty$-categories with $t$-structures.
While these results are relatively straightforward generalizations of Bousfield's pioneering work \cite{bousfield1979localization}, 
we could not locate a reference in the required generality.
These nilpotent completions will be our primary tool throughout the rest of the article.
In \S\ref{sec:rigidity} we prove a variant of Gabber rigidity.
We show that, 
for example, 
if $E \in \SH(X)^\cell$ where $X$ is essentially smooth over a Dedekind scheme, 
then $E(X_x^h)_\ell^\comp \wequi E(x)_\ell^\comp$ for any point $x \in X$ such that $\ell$ is invertible in $k(x)$. 
Here $X_x^h$ denotes the henselization of $X$ along $x$.
Our principal results are shown in \S\ref{sec:models}.
We establish a general method for exhibiting squares as above and provide a criterion for cartesianess in terms of 
étale and real étale cohomology, 
see Proposition \ref{prop:criterion}.
Next we verify this criterion for regular number rings, 
reducing essentially to global class field theory 
---
which is also how Dwyer--Friedlander established \eqref{equation:etalehomotopytype}.
In \S\ref{sec:applications} we discuss some applications, 
including a proof of Theorem \ref{theorem:2mainthmintro}.

\subsection*{Notation and conventions}
We freely use the language of (stable) infinity categories, 
as set out in \cite{lurie-htt,lurie-ha}.
Given a (stable) $\infty$-category $\scr C$ and objects $c, d \in \scr C$, 
we denote by $\Map(c,d) = \Map_{\scr C}(c,d)$ (respectively $\map(c,d) = \map_\scr{C}(c,d)$) the mapping space 
(respectively mapping spectrum).
Given a symmetric monoidal category $\scr C$, 
we denote the unit object by $\1 = \1_\scr{C}$.
We assume familiarity with the motivic stable category $\SH(S)$; 
see e.g., \cite[\S4.1]{bachmann-norms}.
We write $\Sigma^{p,q} = \Sigma^{p-q} \wedge \Gmp{q}$ for the bigraded suspension functor and $S^{p,q} = \Sigma^{p,q} \1$ 
for the bigraded spheres.

\subsection*{Acknowledgements}
We acknowledge the support of the  Centre for Advanced Study at the Norwegian Academy of Science and Letters in Oslo,
Norway, which funded and hosted our research project ``Motivic Geometry" during the 2020/21 academic year,
and the RCN Frontier Research Group Project no. 250399 ``Motivic Hopf Equations."

\section{Nilpotent completions} 
\label{sec:nilpotent-completion}
We axiomatize some well-known facts about nilpotent completions in presentably symmetric monoidal stable $\infty$-categories 
with a $t$-structure.
Our arguments are straightforward generalizations of \cite{bousfield1979localization} and \cite{mantovani2018localizations}.
Theorems \ref{thm:E-vs-pi0E-completion} and \ref{thm:1/x-comp} are the main results in this section.

\subsection{Overview}
Throughout we let $\scr C$ be a presentably symmetric monoidal $\infty$-category 
(i.e., the tensor product preserves colimits in each variable separately) 
provided with a $t$-structure which is compatible with the symmetric monoidal structure 
(i.e., $\scr C_{\ge 0} \otimes \scr C_{\ge 0} \subset \scr C_{\ge 0}$) and \emph{left complete} 
(i.e., for $X \in \scr C$ we have $X \wequi \lim_n X_{\le n}$).
Given $E \in \CAlg(\scr C)$ and $X \in \scr C$ recall \cite[Construction 2.7]{mathew2017nilpotence} 
the standard cosimplicial resolution (or \emph{cobar construction}) 
\[ 
\Delta_+ \to \scr C, [n] \mapsto X \otimes E^{\otimes n+1} 
\] 
whose limit is (for us by definition) the \emph{$E$-nilpotent completion} $X_E^\comp$.

We call $X \in \scr C$ \emph{connective} if $X \in \cup_n \scr C_{\ge n}$.
Recall that $R \in \CAlg(\scr C^\heart)$ is called \emph{idempotent} if the multiplication map 
$R \otimes^\heart R \to R \in \scr C^\heart$ is an equivalence.

\begin{theorem} 
\label{thm:E-vs-pi0E-completion}
Let $\scr C$ be left complete, $E \in \CAlg(\scr C_{\ge 0})$ and $X \in \scr C$.
Suppose that $\pi_0 E \in \CAlg(\scr C^\heart)$ is idempotent and $X$ is connective.
Then the canonical map 
\[ 
X_E^\comp \to X_{\pi_0 E}^\comp 
\] 
is an equivalence.
\end{theorem}

One way of producing idempotent algebras is by taking quotients of the unit.
Given $L_1, \dots, L_n \in \scr C_{\ge 0}$ and maps $x_i: L_i \to \1$, 
we set 
\[ X/(x_1^{m_1}, x_2^{m_2}, \dots, x_n^{m_n}) 
= 
X \otimes \cof(x_n^{\otimes m_n}: 
L_n^{\otimes m_n} \to 1) \otimes \dots \otimes \cof(x_1^{\otimes m_1}: L_1^{\otimes m_1} \to 1) 
\]
The object $\pi_0(\1/(x_1, \dots, x_n)) \in \CAlg(\scr C^\heart)$ is idempotent.
For varying $m$,
the $\1/x_i^m$'s form an inverse system indexed on $\N$ in an evident way; 
by taking tensor products, 
the objects $X/(x_1^{m_1}, \dots, x_n^{m_n})$ form an $\N^n$-indexed inverse system.
We define the $x$-completion of $X$ as the limit
\[ 
X_{x_1, \dots, x_n}^\comp 
:=
\lim_{m_1, \dots, m_n} X/(x_1^{m_1}, \dots, x_n^{m_n})
\]

\begin{theorem} \label{thm:1/x-comp}
Suppose each $L_i\in \scr C_{\ge 0}$ is strongly dualizable with dual $DL_i \in \scr C_{\ge 0}$.
If $X \in \scr C$ is connective and $\scr C$ is left complete, 
then there is a canonical equivalence 
\[ 
X_{\pi_0(\1/(x_1, \dots, x_n))}^\comp 
\wequi 
X_{x_1, \dots, x_n}^\comp
\]
\end{theorem}

To apply Theorem \ref{thm:1/x-comp} in motivic stable homotopy theory we consider, 
for a scheme $S$, 
the \emph{homotopy $t$-structure} on $\SH(S)$; 
see e.g., \cite[\S B]{bachmann-norms}, \cite[\S1]{schmidt2018stable}.

\begin{theorem} 
\label{thm:HFp-completion}
Let $S$ be a quasi-compact quasi-separated scheme of finite Krull dimension and suppose $X \in \SH(S)$ is connective.
\begin{enumerate}
\item There is an equivalence  $X_\MGL^\comp \wequi X_\eta^\comp$.
\item If $1/\ell \in S$ then there is an equivalence $X_{H\F_\ell}^\comp \wequi X_{\eta,\ell}^\comp$.
\end{enumerate}
\end{theorem}
\begin{proof}
The homotopy $t$-structure is left complete by 
\cite[Corollary 3.8]{schmidt2018stable}.\footnote{This reference assumes $S$ noetherian, 
but this is only used to obtain finite homotopy dimension of the Nisnevich topoi, 
which holds in the stated generality by \cite[Theorem 3.17]{clausen2019hyperdescent} --- 
see also \cite[Theorem 4.1]{rosenschon2000descent}.}

(1) Owing to \cite[Theorem 3.8, Corollary 3.9]{hoyois-algebraic-cobordism} we have 
$\MGL \in \SH(S)_{\ge 0}$ and $\pi_0(\MGL) \wequi \pi_0(\1/\eta)$.

(2) We need to prove that $\mathrm{H}\F_\ell \in \SH(S)_{\ge 0}$ and 
$\pi_0(\mathrm{H}\F_\ell) \wequi \pi_0(\1/(\eta,\ell))$.
Since $x_i\in\pi_{2i,i}\MGL$ and $\Sigma^{2i,i}\MGL=\Sigma^i\Gmp{i}\wedge\MGL\in\SH(S)_{\ge i}\subset\SH(S)_{>0}$, 
both of these claims follow from the Hopkins--Morel isomorphism 
\[ 
\mathrm{H}\F_\ell 
\wequi 
\MGL/(\ell, x_1, x_2, \dots)
\] 
shown in 
\cite[Theorem 10.3]{spitzweck2012commutative}.\footnote{This reference assumes $S$ noetherian, 
but since the equivalence exists over $\Z[1/\ell]$ it persists after pullback to $S$.} 
\end{proof}

\begin{remark}
Theorem \ref{thm:HFp-completion} implies that a map $\alpha: E \to F \in \SH(S)_{\ge 0}$ is an $(\eta,\ell)$-adic equivalence 
if and only if $\alpha \wedge \mathrm{H}\F_\ell$ is an equivalence, 
which is also easily seen by considering homotopy objects.
This weaker statement, however, cannot be used as a replacement for Theorem \ref{thm:HFp-completion} in this work.
\end{remark}

\subsection{Proofs}
Recall that $\scr C$ is a presentably symmetric monoidal $\infty$-category equipped with a compatible $t$-structure.
%We will explicitly state left-completeness when that notion is used.

\begin{definition}
\begin{enumerate}
\item Let $E \in \CAlg(\scr C)$.
Then $X \in \scr C$ is \emph{$E$-nilpotent} if it lies in the thick subcategory generated by objects of the form 
$E \otimes Y$ for $Y \in \scr C$.
\item Let $R \in \CAlg(\scr C^\heart)$ be idempotent.
Then $F \in \scr C^\heart$ is \emph{strongly $R$-nilpotent} if $F$ admits a finite filtration whose subquotients 
are $R$-modules.\footnote{Note that $R$ being idempotent is a property, not additional data.}
Moreover, 
$X \in \scr C$ is strongly $R$-nilpotent if it is bounded in the $t$-structure and all homotopy objects are 
strongly $R$-nilpotent.
\end{enumerate}
\end{definition}

\begin{example} \label{ex:htpy-E-mod-nilp}
If $X \in \scr C$ is an $E$-module in the homotopy category, then it is a summand of $X \otimes E$, and thus $X$ is $E$-nilpotent.
\end{example}

\begin{lemma} 
\label{lemm:strongly-R-nilp-basics}
Suppose $R \in \CAlg(\scr C^\heart)$ is idempotent.
\begin{enumerate}
\item 
Let 
\[ A \to B \to C \to D \to E \in \scr C^\heart \] 
be an exact sequence.
If $A,B,D,E$ are strongly $R$-nilpotent, then so is $\scr C$.
\item An object $X \in \scr C$ is strongly $R$-nilpotent if and only if it is $R$-nilpotent and bounded in the $t$-structure.
\end{enumerate}
\end{lemma}
\begin{proof}
(1) The proofs of \cite[Lemmas 7.2.7--7.2.9]{mantovani2018localizations} apply unchanged.
(2) Example \ref{ex:htpy-E-mod-nilp} implies that strongly $R$-nilpotent objects are $R$-nilpotent, 
being finite extensions of homotopy $R$-modules.
It thus suffices to show that if $X$ is $R$-nilpotent, then its homotopy objects $\pi_i^{\scr C}(X) \in \scr C^\heart$ are strongly $R$-nilpotent.
This is clear for free $R$-modules and the property is preserved by taking summands and shifts and cofibers by (1).
The result follows.
\end{proof}

\begin{definition}
\begin{enumerate}
\item If $E \in \CAlg(\scr C)$, $X \in \scr C$, a tower of the form
\[ X \to \dots \to X_2 \to X_1 \to X_0 \] 
is called an \emph{$E$-nilpotent resolution} if each $X_i$ is $E$-nilpotent and for every $E$-nilpotent $Y \in \scr C$, 
we have 
\[ 
\colim_n [X_n, Y] \xrightarrow{\wequi} [X, Y] 
\]
\item 
If $R \in \CAlg(\scr C^\heart)$ is idempotent and $X \in \scr C$, a tower of the form 
\[ X \to \dots \to X_2 \to X_1 \to X_0 \] 
is called a \emph{strongly $R$-nilpotent resolution} if each $X_i$ is strongly $R$-nilpotent and for every strongly 
$R$-nilpotent $Y \in \scr C$, we have 
\[ 
\colim_n [X_n, Y] \xrightarrow{\wequi} [X, Y] 
\]
\end{enumerate}
\end{definition}

\begin{proposition} 
\label{prop:resn-unique}
For $X, Y \in \scr C$ and $X_\bullet, Y_\bullet$ $E$-nilpotent (respectively strongly $R$-nilpotent) resolutions, 
we have 
\[ 
\Map_{\Pro(\scr C)}(X_\bullet, Y_\bullet) \wequi \lim_n \Map(X, Y_\bullet) 
\]
Thus any map $X \to Y$ induces a canonical morphism of towers $X_\bullet \to Y_\bullet$.
In particular, 
if $X \wequi Y$, 
then $X_\bullet \wequi Y_\bullet \in \Pro(\scr C)$ and $\lim_n X_n \wequi \lim_n Y_n$.
\end{proposition}
\begin{proof}
Essentially by definition we have 
\[ 
\Map(X_\bullet, Y_\bullet) \wequi \lim_n \colim_m \Map(X_m, Y_n) 
\] 
The colimit is equivalent to $\Map(X, Y_n)$ by the definition of a resolution.
\end{proof}

\begin{lemma} 
\label{lemm:build-resolution}
Let $E \in \CAlg(\scr C)$ and $X \in \scr C$.
\begin{enumerate}
\item The tower of partial totalizations of the standard cosimplicial objects $X \otimes E^{\otimes \bullet}$ is an $E$-nilpotent resolution of $X$.
\item Suppose that $E \in \scr C_{\ge 0}$ and $\pi_0 E$ is idempotent.
Then if $X \to X_\bullet$ is any $E$-nilpotent resolution by connective objects 
(e.g., if $X$ is connective, the one arising from (1)), 
then $X \to \tau_{\le \bullet} X_\bullet$ is a strongly $\pi_0(E)$-nilpotent resolution.
\end{enumerate}
\end{lemma}
\begin{proof}
(1) Since partial totalizations are finite limits they commute with $\otimes X$, 
by stability,  
and are thus given by $X_i = X \otimes \cof(I^{\otimes i} \to \1)$, 
where $I = \fib(\1 \to E)$, 
see \cite[Proposition 2.14]{mathew2017nilpotence}.
In the notation of loc. cit. we get $\cof(X_i \to X_{i-1}) \wequi \Sigma \cof(T_i(E, X) \to T_{i-1}(E, X))$ and $X_0 = 0$. 
This implies $X_i$ is $E$-nilpotent by \cite[Proposition 2.5(1)]{mathew2017nilpotence}.
To conclude, 
it suffices to prove that if $Y$ is $E$-nilpotent, 
then $\colim_i \map(X_i, Y) \wequi \map(X, Y)$.
The class of objects $Y$ satisfying the latter equivalence is thick, so we may assume that $Y$ is an $E$-module.
We are reduced to proving that $\colim_i \map(I^{\otimes i} \otimes X, Y) = 0$.
But this is a summand of $\colim_i \map(I^{\otimes i} \otimes X \otimes E, Y)$, $Y$ being an $E$-module, 
and the transition maps $I^{\otimes i+1} \otimes E \to I^{\otimes i} \otimes E$ are null by 
\cite[Proposition 2.5(2)]{mathew2017nilpotence}, so the colimit vanishes as desired.

(2) We first show that each $\tau_{\le n} X_n$ is strongly $R$-nilpotent, 
and more generally, 
that if $Y$ is $E$-nilpotent then each $\pi_i(Y)$ is strongly $R$-nilpotent.
By Lemma \ref{lemm:strongly-R-nilp-basics}(1) we may assume $Y$ is a (free) $E$-module; 
in this case, each $\pi_i(Y)$ is a $\pi_0(E)$-module.\NB{ref}
Suppose $Y \in \scr C$ is strongly $\pi_0(E)$-nilpotent.
Then $Y$ is $E$-nilpotent since any $\pi_0(E)$-module is an $E$-module.
Finally, we have 
\[ 
\colim_n [\tau_{\le n} X_n, Y] \wequi \colim_n [X_n, Y] \wequi [X, Y] 
\] 
Here the first equivalence holds since $Y$ is bounded above and the second because $Y$ is $E$-nilpotent.
\end{proof}

Next we prove that the $E$-nilpotent completion only depends on $\pi_0(E)$.

\begin{proof}[Proof of Theorem \ref{thm:E-vs-pi0E-completion}.]
For $E \in \CAlg(\scr C_{\ge 0})$ and $X \in \scr C$, 
denote by $R_n(E, X)$ the $n$-th partial totalization of $X \otimes E^{\otimes \bullet}$, 
so that $X \to R_\bullet(E, X)$ is a tower with limit $X \to X_E^\comp$.
By left completeness and cofinality we have 
\[ 
X_E^\comp \wequi \lim_{m,n} \tau_{\le m} R_n(X, E) \wequi \lim_n \tau_{\le n} R_n(X, E) 
\]
By Lemma \ref{lemm:build-resolution}, 
the right hand side is the limit of a strongly $\pi_0(E)$-nilpotent resolution, 
which by Proposition \ref{prop:resn-unique} only depends on $X$ and $\pi_0(E)$.
\end{proof}

\begin{remark} 
\label{rmk:ampli}
The proof also verifies that any strongly $\pi_0(E)$-nilpotent resolution of $X$ has limit $X_{\pi_0 E}^\comp$.
\end{remark}

We now turn to the study of $x$-completions.

\begin{lemma} 
\label{lemm:x-comp-against-null}
Let $L_1, \dots, L_n \in \scr C$ be strongly dualizable and $x_i: L_i \to \1$.
Let $Y \in \scr C$ and suppose that, 
for every $i$, 
the map \[ Y \otimes L_i \xrightarrow{x_i} Y \] is null.
Then there is an equivalence
\[ 
\colim_{m_1, \dots, m_n} \map(X/(x_1^{m_1}, \dots, x_n^{m_n}), Y) \wequi \map(X, Y) 
\]
\end{lemma}
\begin{proof}
As a first observation, 
note that the maps 
$Y \otimes L_i \xrightarrow{x_i} Y$ and $Y \xrightarrow{Dx_i} Y \otimes DL_i$ correspond
under the equivalence $\Map(Y \otimes L_i, Y) \wequi \Map(Y, Y \otimes D(L_i))$.
It follows that $Dx_i$ is null.

First consider the case $n=1$.
By definition we have $\fib(X \to X/x^m) \wequi X \otimes L^{\otimes m}$.
Hence it suffices to prove $\colim_m \map(X \otimes L^{\otimes m}, Y) = 0$.
This term can be identified with $\colim_m \map(X, (DL)^{\otimes m} \otimes Y)$, 
and the transition maps in this system are null by our first observation.
In the general case, 
we note the equivalence
\[ 
X/(x_1^{m_1}, \dots, x_n^{m_n}) \wequi (X/x_1^{m_1})/(x_2^{m_2}, \dots, x_n^{m_n}) 
\]
Hence we get
\begin{align*}
\colim_{m_1, \dots, m_n} \map(X/(x_1^{m_1}, \dots, x_n^{m_n}), Y)
  &\wequi \colim_{m_1} \colim_{m_2, \dots, m_n} \map((X/x_1^{m_1})/(x_2^{m_2}, \dots, x_n^{m_n}), Y) \\
  &\wequi \colim_{m_1} \map(X/x_1^{m_1}, Y) \\
  &\wequi \map(X, Y)
\end{align*}
The first equivalence holds since colimits commute, and the other two hold by induction.
\end{proof}

\begin{lemma} 
\label{lemm:dualizable-pi}
Suppose $L \in \scr C_{\ge 0}$ is strongly dualizable with strong dual $DL \in \scr C_{\ge 0}$.
Then, 
for all $X \in \scr C$,
there are equivalences
\[
\pi_i(X\otimes L) 
\wequi \pi_i(X)\otimes L
\wequi \pi_i(X)\otimes^\heart \pi_0(L)
\]
\end{lemma}
\begin{proof}
By assumption we have $\scr C_{\ge 0} \otimes L \subset \scr C_{\ge 0}$.
The same holds for $DL$, which implies $\scr C_{\le 0} \otimes L \subset \scr C_{\le 0}$.
In other words, $\otimes L: \scr C \to \scr C$ is $t$-exact, and hence $\pi_i(X \otimes L) \wequi \pi_i(X) \otimes L$.
Being in the heart $\scr C^\heart$, 
the latter tensor product is equivalent to $\pi_i(X) \otimes^\heart \pi_0(L)$.
\end{proof}

Let us quickly verify that $\pi_0(\1/(x_1, \dots, x_n))$ is indeed an idempotent algebra in $\scr C^\heart$.

\begin{lemma} 
\label{lemm:produce-idempotent}
Let $L_1, \dots, L_n \in \scr C_{\ge 0}$ and $x_i: L_i \to \1$.
Then $R = \pi_0(\1/(x_1, \dots, x_n))$ defines an idempotent object of $\CAlg(\scr C^\heart)$ and the 
multiplication maps $\pi_0(L_i) \otimes^\heart R \xrightarrow{x_i} R$ are null.
\end{lemma}
\begin{proof}
Recall that idempotent commutative algebras in $\scr C^\heart$ are the same as maps $\pi_0(\1) \to A \in \scr C^\heart$ 
such that the induced map $A \to A \otimes^\heart A$ is an isomorphism \cite[Proposition 4.8.2.9]{lurie-ha}.
Note that 
\[ 
\pi_0(\1/(x_1, \dots, x_n)) \wequi \pi_0(\pi_0(\1/(x_1, \dots, x_{n-1}))/x_n) 
\]
More generally, 
let us prove that if $\pi_0(\1) \to A \in \scr C^\heart$ is an idempotent algebra and $L \in \scr C_{\ge 0}$, 
$x: L \to \1$, 
then $\pi_0(A/x)$ is also an idempotent algebra on which multiplication by $x$ is null.
Consider the commutative diagram of cofiber sequences
\begin{equation*}
\begin{CD}
L \otimes A \otimes A/x @>e>> A \otimes A/x @>u>> A/x \otimes A/x \\
@AdAA                           @AbAA                @AAA         \\
L \otimes A \otimes A   @>c>> A \otimes A   @>a>> A/x \otimes A
\end{CD}
\end{equation*}
Here $c$ and $e$ ``multiply $L$ into the left factor $A$'', and all the other maps are the canonical projections.
Since $A$ is idempotent, 
$\pi_0(A \otimes A) \wequi A$ and $\pi_0(A \otimes A/x) \wequi \pi_0(A/x) \wequi \pi_0(A/x \otimes A)$.
Under these identifications we have $\pi_0(a) = \pi_0(b)$ and so $\pi_0(ed) = \pi_0(bc) = \pi_0(ac) = 0$.
Since $\pi_0(d)$ is an epi we deduce $\pi_0(e) = 0$, 
and hence $\pi_0(u)$ is an isomorphism.
This concludes the proof since, 
under our identifications, 
$\pi_0(e)$ is multiplication by $x$ on $\pi_0(A/x)$ and $\pi_0(u)$ is $\pi_0(A/x)\to\pi_0(A/x)\otimes^\heart\pi_0(A/x)$.
\end{proof}

We can now identify $x$-completions as $E$-nilpotent completions for an appropriate $E$.

\begin{proof}[Proof of Theorem \ref{thm:1/x-comp}.]
Lemma \ref{lemm:produce-idempotent} shows $R_n = \pi_0(\1/(x_1, \dots, x_n))$ is idempotent.

\emph{Step 1}: The map $R_n \otimes L_i \xrightarrow{x_i} R_n$ is null.
Indeed, by Lemma \ref{lemm:dualizable-pi}, 
we have $R_n \otimes L_i \wequi R_n \otimes^\heart \pi_0(L_i)$, 
and so this follows from Lemma \ref{lemm:produce-idempotent}.

\emph{Step 2}: We show the homotopy objects of $X/(x_1^{e_1}, \dots, x_n^{e_n})$ are strongly $R_n$-nilpotent 
for all $e_i \ge 1$.
By an induction argument, 
using the octahedral axiom, 
$X/x^m$ is a finite extension of copies of $X/x$.
Hence each $X/(x_1^{e_1}, \dots, x_n^{e_n})$ is a finite extension of copies of $X/(x_1, \dots, x_n)$;
thus we may assume $e_i=1$.
By induction on $n$ and Lemma \ref{lemm:dualizable-pi}, 
together with Lemma \ref{lemm:strongly-R-nilp-basics}(1), 
it suffices to show that if $M \in \scr C^\heart$ is $R_i$-nilpotent, 
then both the kernel and cokernel of 
\[ M \otimes^\heart \pi_0(L_{i+1}) \xrightarrow{x_{i+1}} M \] 
are $R_{i+1}$-nilpotent.
The proof given in \cite[Lemma 7.2.10]{mantovani2018localizations} goes through unchanged in our setting.

\emph{Step 3}: We show that 
\[ 
\{\tau_{\le m} X/(x_1^{e_1}, \dots, x_n^{e_n}) \}_{e_1, \dots, e_n;m} 
\] 
is a strongly $R_n$-nilpotent resolution of $X$.
Since we assume $X$ is connected, 
step 2 shows 
\[
\tau_{\le m} X/(x_1^{e_1}, \dots, x_n^{e_n})
\] 
is bounded with strongly $R_n$-nilpotent homotopy objects.
Owing to Lemma \ref{lemm:strongly-R-nilp-basics}(2) it is in fact strongly $R_n$-nilpotent.
We thus need to show that if $Y$ is strongly $R_n$-nilpotent, 
then 
\[ 
\colim \map(\tau_{\le m} X/(x_1^{e_1}, \dots, x_n^{e_n}), Y) 
\wequi \map(X, Y) 
\]
Since $Y$ is bounded above, we may remove $\tau_{\le m}$ in the above expression without changing the colimit.
We may assume that $Y$ is an $R_n$-module in $\scr C^\heart$.
By step 1 the map $L_i \otimes Y \to Y$ is null, 
and so the claim follows from Lemma \ref{lemm:x-comp-against-null}.

\emph{Conclusion of proof}:
By left completeness we have 
\[ 
X_{x_1, \dots, x_n}^\comp 
\wequi \lim_{e_1, \dots, e_n; m} \tau_{\le m} X/(x_1^{e_1}, \dots, x_n^{e_n}) 
\]
According to step 3, 
this is the limit of a strongly $R_n$-nilpotent resolution of $X$, 
which coincides with $X_{R_n}^\comp$ by Remark \ref{rmk:ampli}.
\end{proof}

\section{Rigidity for stable motivic homotopy of henselian local schemes} 
\label{sec:rigidity}

Given a presentably symmetric monoidal stable $\infty$-category $\scr C$ and a morphism 
$a: L \to \1$ with $L$ strongly dualizable, 
we denote by $\scr C_a^\comp$ the $a$-completion; 
that is, 
the localization at maps which become an equivalence after $\otimes \cof(a)$.
We refer to \cite[\S2.1]{bachmann-SHet}, \cite[\S2.5]{bachmann-eta} for more details; 
in particular, 
the $a$-completion of $X$ is given by the object $X_a^\comp$ from the previous section.

Given a family of objects $\scr G \subset \scr C$ (which for us will always be bigraded spheres $\Sigma^{**} \1$), 
we write $\scr C^\cell$ for the localizing subcategory generated by $\scr G$.
Noting that $\scr C_a^\comp$ is equivalent to the localizing tensor ideal generated by $\cof(a)$, 
by e.g., \cite[Example 2.3]{bachmann-SHet}, 
we see that if $L \in \scr G$ then these two operations commute, 
and so we shall write 
\[ 
\scr C_a^{\comp\cell} := (\scr C_a^\comp)^\cell \wequi (\scr C^\cell)_a^\comp 
\]

Recall the element $h := 1 + \lra{-1} \in \pi_{0,0}(\1)$, 
where $-\lra{-1}$ is the switch map on $\Gm \wedge \Gm$, 
and the element $\rho := [-1] \in \pi_{-1,-1}(\1)$ corresponding to $-1 \in \scr O^\times$.

\begin{proposition} 
\label{prop:rigidity}
Suppose $X$ is a henselian local scheme and essentially smooth over a Dedekind scheme.
Write $i:x \to X$ for the inclusion of the closed point and let $n \in \Z$.
\begin{enumerate}
\item If $1/n \in X$ then $i^*: \SH(X)_{n}^{\comp\cell} \to \SH(x)_{n}^{\comp\cell}$ is an equivalence.
\item If $1/2n \in X$ then $i^*: \SH(X)_{nh}^{\comp\cell} \to \SH(x)_{nh}^{\comp\cell}$ is an equivalence.
\item $i^*: \SH(X)[\rho^{-1}]^\cell \to \SH(x)[\rho^{-1}]^\cell$ is an equivalence.
\end{enumerate}
\end{proposition}
Many proofs in the sequel will follow the pattern of this one.
We spell out many details here, which are suppressed in the following proofs.
\begin{proof}
If $S$ is a quasi-compact quasi-separated scheme, 
e.g., affine, 
the category $\SH(S)$ is compactly generated by suspension spectra of finitely presented smooth $S$-schemes 
\cite[Proposition C.12]{hoyois2015quadratic}.
Thus $\SH(S)^\cell$ is compactly generated by the spheres, 
and for every $a \in \pi_{**}(\1_S)$, 
the category $\SH(S)_a^{\comp\cell}$ is compactly generated by $\Sigma^{**}\1/a$.
Now let $f: S' \to S$ be a morphism, where $S'$ is also quasi-compact quasi-separated.
We use $f^*$ to transport elements of $\pi_{**}(\1_S)$ to $\pi_{**}(\1_{S'})$, 
and when no confusion can arise, 
we denote them by the same letter.
Thus, for example, we set 
\[ 
\SH(S')_a^\comp := \SH(S')_{f^*a}^\comp 
\]
The functor $f^*: \SH(S)_a^{\comp\cell} \to \SH(S')_a^{\comp\cell}$ preserves colimits and the compact generator. 
Therefore it admits a right adjoint $f_*$ preserving colimits.
This implies that $f^*$ is fully faithful if and only if the map $\1 \to f_*f^*\1 \in \SH(S)_a^{\comp\cell}$ is an equivalence,  
see e.g., \cite[Lemma 22]{bachmann-hurewicz}; 
in this case, the functor is an equivalence since its essential image will be a localizing subcategory containing the generator.

We can simplify this condition further.
By $a$-completeness and Lemma \ref{lemm:duality-yoga} below, 
it follows that $\1 \to f_*f^*\1$ is an equivalence if and only if $\1/a \to f_*f^*(\1/a)$ is an equivalence, 
i.e., if and only if 
\[ 
\pi_{**}(\1_S/a) \wequi \pi_{**}(\1_{S'}/a)
\]

If $b \in \pi_{**}(\1)$, 
then in our compactly generated situations the $b$-periodization $\mathcal{E}[b^{-1}]$ is given by the colimit 
\[ 
\mathcal{E}[b^{-1}] = \colim \left(\mathcal{E} \xrightarrow{b} \Sigma^{**} \mathcal{E} \xrightarrow{b} \dots \right)
\]
Since $f_*$ preserves colimits it commutes with $b$-periodization by Lemma \ref{lemm:duality-yoga}.
We shall make use of the fact that a map is an equivalence if and only if it is an equivalence after $b$-periodization 
and $b$-completion, 
see e.g., \cite[Lemma 2.16]{bachmann-eta}.
Thus to prove fully faithfulness it would also be sufficient, as well as necessary, to prove 
\[ 
\pi_{**}(\1_S/(a,b)) 
\wequi 
\pi_{**}(\1_{S'}/(a,b)) 
\quad\text{and}\quad 
\pi_{**}(\1_S[b^{-1}]/a) 
\wequi 
\pi_{**}(\1_{S'}[b^{-1}]/a)
\]
We will use many different variants of these observations in the sequel.

(0)
We claim the functor
\[ 
\SH(X)[\eta^{-1}] \to \SH(x)[\eta^{-1}] 
\] 
is an equivalence provided $1/2 \in X$, and that 
\[  
\SH(X)[\eta^{-1},1/2] \to \SH(x)[\eta^{-1},1/2] 
\] 
is an equivalence without any assumptions on $X$.
For the first claim, 
by the above remarks it suffices to prove that $\pi_{**}(\1[\eta^{-1}])$ satisfies the required rigidity, 
which via \cite[Proposition 5.2]{bachmann-etaZ} reduces to the same statement for the Witt ring $W(\ph)$.
This is true by \cite[Lemma 4.1]{jacobson2018cohomological}.
Since $\SH(S)[\eta^{-1}, 1/2] \wequi \SH(S)[\rho^{-1}, 1/2]$, 
see Lemma \ref{lemm:eta-basics}, 
the second claim reduces to (3).

(1)
It suffices to establish an isomorphism on $\eta$-periodization and $\eta$-completion.
We first treat the $\eta$-complete case, 
i.e.,  
we need to show that $\1 \to i_*i^* \1 \in \SH_{n,\eta}^{\comp\cell}$ is an equivalence.
%i.e. we wish to show that \[ \pi_{**}(\1_{X,n,\eta}^\comp) \wequi \pi_{**}(\1_{x,n,\eta}^\comp). \]
By Theorem \ref{thm:HFp-completion}(2) with $\ell=n$, 
we have 
\[ E_{n,\eta}^\comp \wequi \lim_\Delta E \wedge H\Z/n^{\wedge \bullet+1} \] 
for any connective $E$ in $\SH(S)$.
The cellularization functor $\SH(S) \to \SH(S)^\cell$ preserves limits and hence $(n,\eta)$-completions.
Moreover, 
$\mathrm{H}\Z/n \in \SH(S)^\cell$ if $1/n \in S$ by \cite[Corollary 10.4]{spitzweck2012commutative}.
Hence the above formula for $E_{n,\eta}^\comp$ also makes sense, 
and is true, 
in $\SH(S)^\cell$.
%In particular a map $E \to F \in \SH(S)^\cell$ between connective spectra is an equivalence if and only if $E \wedge H\Z/n \to F \wedge H\Z/n$ is.
Thus we need to show the map $\mathrm{H}\Z/2^{\wedge t} \to i_*(\mathrm{H}\Z/2^{\wedge t}) \in \SH(X)^\cell$ is an equivalence, 
for $t \ge 1$.
Lemma \ref{lemm:duality-yoga} below implies that $i_*(E \wedge i^*F) \wequi i_*(E) \wedge F$, 
for any $E \in \SH(x)$, $F \in \SH(X)^\cell$.
In this way we reduce to $t=1$; 
i.e., 
it suffices to show 
\[ 
\pi_{**}(\mathrm{H}\Z/n_X) \wequi \pi_{**}(\mathrm{H}\Z/n_x)
\]
Owing to \cite[Theorem 3.9]{spitzweck2012commutative}, 
$\pi_{**}(\mathrm{H}\Z/n_S)$ is given by the Zariski cohomology of $S$ with coefficients in a truncation of the 
étale cohomology of $\mu_n^{\otimes \ph}$.
When $S=X$ or $S=x$ the scheme $S$ is Zariski local, so $\pi_{**}(\mathrm{H}\Z/n_S)$ is simply given by certain 
étale cohomology groups of $S$ with coefficients in $\mu_n^{\otimes \ph}$.
The rigidity result follows now from \cite[Theorem 1]{gabber1994affine}.

Next we treat the $\eta$-periodic case.
If $n$ is even then $1/2 \in X$ and so the result follows from (0).
If $n$ is odd then $n$-complete objects are $2$-periodic, and the result also follows from (0).

(2) Again it suffices to prove that we have an isomorphism after $\eta$-completion and $\eta$-periodization; 
(0) handles the $\eta$-periodic case.
For the $\eta$-complete case, 
we use that $\pi_0(\1/(nh,\eta)) \wequi \pi_0(1/(2n,\eta))$, 
see Lemma \ref{lemm:eta-basics} below, 
whence $\1_{nh,\eta}^\comp \wequi \1_{2n,\eta}^\comp$ by Theorem \ref{thm:1/x-comp}; 
this reduces to (1).

(3)
By \cite[Theorem 35]{bachmann-real-etale} we have $\SH(S)[\rho^{-1}] \wequi \SH(S_\ret)$, 
where the right hand side denotes hypersheaves on the small real étale site of $S$.
In this situation we have a natural $t$-structure, 
see e.g., \cite[\S2.2]{bachmann-SHet}, 
such that the map $\1_\ret \to \mathrm{H}_\ret\Z$ is a morphism of connective ring spectra inducing an isomorphism on $\pi_0$
--- where by $\mathrm{H}_\ret\Z$ we mean the constant sheaf of spectra.
Hence, 
applying Theorem \ref{thm:E-vs-pi0E-completion} in this situation, 
and repeating the above discussion using that $\mathrm{H}_\ret\Z$ is cellular and stable under base change, 
essentially by definition, 
we find that in order to prove $\1 \to i_*i^* \1 \in \SH(S)[\rho^{-1}]^\cell$ is an equivalence it suffices to prove
$\mathrm{H}_\ret \Z \to i_* \mathrm{H}_\ret \Z$ is an equivalence.
In other words, 
we need to show 
\[
H^*_\ret(X, \Z) \wequi H^*_\ret(x, \Z)
\]
Since the real étale and Zariski cohomological dimension coincide \cite[Theorem 7.6]{real-and-etale-cohomology}, 
we are reduced to $H^0_\ret$, 
which follows from \cite[Propositions II.2.2, II.2.4]{andradas2012constructible}.
%which is given by the set of continuous functions $C(R\ph,\ph)$.
%This case follows from \cite[Propositions II.2.2, II.2.4]{andradas2012constructible} 
%as in \cite[proof of Lemma 6.4]{karoubi2015witt}.
\end{proof}

\begin{lemma} 
\label{lemm:duality-yoga}
Let $F: \scr C \to \scr D$ be a symmetric monoidal functor between symmetric monoidal categories admitting a right adjoint $G$, 
and let $x: A \to \1$ be a morphism in $\scr C$ with $A$ strongly dualizable.
Then for $X \in \scr D$, 
there is a natural equivalence $G(X \otimes FA) \wequi G(X) \otimes A$, 
and under this equivalence the map 
\[ 
G(X \otimes FA) \xrightarrow{G(\id \otimes Fx)} G(X) 
\] 
corresponds to 
\[ 
GX \otimes A \xrightarrow{x} GX 
\]

Suppose that $\scr C, \scr D$ are presentably symmetric monoidal stable $\infty$-categories and $G$ preserves colimits.
Write $\scr C'$ for the localizing subcategory of $\scr C$ generated by strongly dualizable objects.
Then the above result also holds for any $A \in \scr C'$.
\end{lemma}
\begin{proof}
Since $F$ symmetric monoidal, 
$G$ is lax symmetric monoidal, 
and there is a canonical map $GX \otimes GFA \to G(X \otimes FA)$.
Composing with the unit $A \to GFA$, 
we obtain a natural map $GX \otimes A \to G(X \otimes FA)$, 
which is an equivalence by the Yoneda lemma.
Since this equivalence is natural in $A$ as well, the claim about $x$ also follows.

For the second statement, 
the subcategory comprised of $A \in \scr C$ for which the natural transformation $GX \otimes A \to G(X \otimes FA)$ 
is an equivalence for all $X \in \scr D$ is localizing since $G$ preserves colimits and it contains all strongly 
dualizable objects by the first part, 
hence all of $\scr C'$.
\end{proof}

\begin{lemma} 
\label{lemm:eta-basics}
In $\pi_{*,*}(\1)$ we have the relations
\[
\eta h = 0, h = 2 + \eta \rho, h \rho^2 = 0
\]
It follows that 
$$
\SH(S)[1/2,1/\eta] \wequi \SH(S)[1/2,1/\rho]
$$
\end{lemma}
\begin{proof}
By \cite[Theorem 1.2]{druzhinin2018homomorphism} all the Milnor--Witt relations hold in $\pi_{*,*}(\1)$, 
including $\eta h = 0$.
Our definition of $h$ agrees with Druzhinin's by \cite[Lemma 3.10]{druzhinin2018homomorphism}.
We now compute 
\[ 
h\rho^2 = (2 + \eta[-1])[-1][-1] = 2[-1][-1] + ([1] - [-1] - [-1])[-1] = 0 
\] 
using the logarithm relation $[ab] = [a] + [b] + \eta[a][b]$ as well as $[1]=0$ which holds by definition.

For the last part, 
note that inverting either $\eta$ or $\rho$ kills $h$ (by the first or third relation), 
and hence makes $\eta$ and $\rho$ inverses of each other up to a factor of $-1/2$ (by the second relation).
\end{proof}

\begin{example} 
\label{ex:eta-periodic-h}
Suppose that $1/2 \in X$, where $X$ is henselian local over a Dedekind scheme.
Applying Proposition \ref{prop:rigidity}(2) with $n=1$ we learn that $\SH(X)_h^{\comp\cell} \to \SH(x)_h^{\comp\cell}$ 
is an equivalence.
By Lemma \ref{lemm:eta-basics} both the $\eta$-periodic and $\rho$-periodic objects are $h$-torsion. 
We conclude $\SH(X)_h^\comp[\eta^{-1}] \wequi \SH(X)[\eta^{-1}]$ and similarly for $\rho$.
Thus there is an equivalence 
\[ 
\SH(X)[\eta^{-1}]^\cell \wequi \SH(x)[\eta^{-1}]^\cell 
\] 
A similar equivalence holds for $\rho$.
With reference to Proposition \ref{prop:rigidity}, 
this shows (2) implies (3).
%(but (3) holds more generally than (2)).
\end{example}

\begin{example} 
\label{ex:n-compl-strong}
We have $(E_{ab}^\comp)_a^\comp \wequi E_a^\comp$ since $ab$-periodic objects are $a$-periodic. 
Hence, in  Proposition \ref{prop:rigidity}, (2) implies (1).
%Note, however, that (1) holds more generally.
\end{example}

\todo{removed false claim}
%\begin{remark}
%Proposition \ref{prop:rigidity} remains true with the pair $(X,x)$ replaced by any henselian pair $(X,Z)$, 
%where both $X$ and $Z$ are essentially smooth over not necessarily the same Dedekind scheme.
%To prove this, 
%since $i_*$ commutes with essentially smooth base change, 
%one may assume $X$ henselian local, say with closed point $x$.
%Then $Z$ is also henselian local \cite[Tag 0DYD]{stacks-project}, 
%and the result follows by applying Proposition \ref{prop:rigidity} to both $X$ and $Z$.
%\end{remark}

\section{Topological models for stable motivic homotopy of regular number rings} 
\label{sec:models}
We shall exhibit pullback squares describing $\SH(\mathcal{O}_{F}[1/\ell])_\ell^{\comp\cell}$ for suitable 
number fields $F$ and prime numbers $\ell$ in terms of $\SH(k)_\ell^{\comp\cell}$ for fields of the form 
$k=\C,\R,\F_{q}$.
To facilitate comparison with the work of Dwyer--Friedlander \cite{dwyer1994topological} we formally dualize 
our terminology and exhibit pushout squares in the opposite category.

\subsection{Setup}
Let $\ell$ be a prime (or more generally any integer, but we do not need or use this extra generality).
We shall use the notation $\ell' = \ell$ if $\ell$ is odd, and $\ell' = \ell h$ if $\ell=2$.
\begin{definition} \hfill
\begin{enumerate}
\item We write 
\[ \Premot_S \subset (\CAlg(\PrL)^\op)_{/\SH(S)^\cell} \] 
for the full subcategory comprised of functors $F\colon \SH(S)^\cell \to \scr C$, 
where $\scr C$ is generated under colimits by $F(\SH(S)^\cell)$ (or equivalently by $F(S^{p,q})$ for $p,q \in \Z$).
\item We denote by $M_{\ell'}$ the functor 
\[
\Schl \to \Premot_{\Z[1/\ell]}, \quad X \mapsto \SH(X)_{\ell'}^{\comp\cell}, 
\quad (f: X \to Y) \mapsto (f^*: \SH(Y)_{\ell'}^{\comp\cell} \to \SH(X)_{\ell'}^{\comp\cell})^\op 
\]
\end{enumerate}
\end{definition}
We also put $\Premot = \Premot_{\Z}$ and, by abuse of notation, $M(X) := M_0(X) = \SH(X)^\cell \in \Premot$.
Note that $\Premot_S = \Premot_{/M_0(S)}$ and $M_{\ell'}(X) = M(X)_{\ell'}^\comp$.
Next we clarify the meaning of colimits in $\Premot_S$.
%\NB{relevance of SAG.16.2.0.2?}

\begin{lemma} 
\label{lemm:premot-colim}
Let $F: I \to \Premot_S$ be a diagram and write $F': I^\op \to \Cat_\infty$ for the underlying diagram of categories.
Then $\lim_{I^\op} F' \in \Cat_\infty$ is presentably symmetric monoidal and admits a natural functor from $\SH(S)^\cell$.
Let $\scr C$ denote its subcategory generated under colimits by the image of $\SH(S)^\cell$.
Then there is an equivalence $\colim_I F \wequi \scr C$.
\end{lemma}
\begin{proof}
The forgetful functor 
\[
\CAlg(\PrL)_{\SH(S)^\cell/} \to \Cat_\infty
\]
preserves limits \cite[Propositions 5.5.3.13, 1.2.13.8]{lurie-htt}, 
\cite[Corollary 3.2.2.5]{lurie-ha}, 
and hence the limit admits a canonical functor from $\SH(S)^\cell$.
For $\scr D \in \Premot$ we have 
\begin{gather*} 
\Map_{\Premot}(\scr C, \scr D) \wequi 
\Map_{\CAlg(\PrL)_{\SH(S)^\cell/}}(\scr D, \scr C) 
\subset 
\Map_{\CAlg(\PrL)_{\SH(S)^\cell/}}(\scr D, \lim_{I^\op} F') \\ 
\wequi 
\lim_{I^\op} \Map_{\CAlg(\PrL)_{\SH(S)^\cell/}}(F(\ph), \scr D) 
\wequi 
\lim_{I^\op} \Map_{\Premot}(F(\ph), \scr D)
\end{gather*} 
It remains to show the inclusion is an equivalence, 
i.e., 
every map $\scr D \to \lim_{I^\op} F'$ in $\CAlg(\PrL)_{\SH(S)^\cell/}$ factors through $\scr C$.
This holds for the generators, 
by assumption, 
so we are done.
\end{proof}

Next we reformulate and slightly extend our rigidity results from \S\ref{sec:rigidity}.

\begin{lemma}
\label{lemm:synpt-basics}
Let $\bar{x}$ be the spectrum of a separably closed field, 
$X \in \Schl$ an essentially smooth over a Dedekind domain, 
$\bar x \to X$ a map, 
and $y \in X$ a specialization of the image of $x$.
In $\Premot_{\Z[1/\ell]}$ there is a commutative diagram 
\begin{equation*}
\begin{tikzcd}
M(\bar x) \ar[d] \ar[r, "s"] & M(y) \ar[ld] \\
M(X)
\end{tikzcd}
\end{equation*}
Here the unlabelled maps are the canonical ones.
In fact, 
there is a family of such commutative diagrams, 
parametrized by the (non-empty) set $X_y^h \times_X \bar x$.
\end{lemma}
\begin{proof}
Let $X'$ be the henselization of $X$ along $y$.
By \cite[Tags 03HV, 07QM(1)]{stacks-project} the map $X' \to X$ hits the image of $\bar x$, 
and hence there exists a lift $s'$ in the commutative diagram
\begin{equation*}
\begin{tikzcd}
\bar x \ar[d] \ar[r, "s'"] & X' \ar[ld] \\
X
\end{tikzcd}
\end{equation*}
Applying $M$ and using that $M(y) \to M(X')$ is an equivalence by Proposition \ref{prop:rigidity}, the result follows.
\end{proof}

\begin{corollary}
The following hold under the assumptions in Lemma \ref{lemm:synpt-basics}.
\begin{enumerate}
\item If $y\in X$ is separably closed, then $s$ is an equivalence.
\item If $\bar x, \bar y \in \Schl$ are separably closed fields there is a (non-unique) equivalence 
$M(\bar x) \wequi M(\bar y)$.
\end{enumerate}
\end{corollary}
\begin{proof}
(1) We have constructed a symmetric monoidal cocontinuous functor 
$F: \SH(\bar y)_{\ell'}^{\comp\cell} \to \SH(\bar x)_{\ell'}^{\comp\cell}$ under $\SH(\Z[1/\ell])^\cell$.
Denote its right adjoint by $G$.
Arguing as in the proof of Proposition \ref{prop:rigidity}, 
it suffices to show $\mathcal{E} \to GF \mathcal{E}$ is an equivalence. 
That is, 
$\mathcal{E} \to GF \mathcal{E}$ induces an isomorphism on $\pi_{**}$ for $\mathcal{E} = \mathrm{H}\F_\ell$, 
$\mathcal{E} = \mathrm{H}_\ret \Z$ and, 
if $\ell$ is even, $\mathcal{E} = \1[\eta^{-1}]$.
For any separably closed field of characteristic $\ne \ell$ we have $\pi_{**}(\mathrm{H}\F_\ell) \wequi \F_\ell[\tau]$, 
see e.g., \cite[Corollary C.2(2)]{BKWX}, \cite[Theorem 18.2.7]{io-survey}, 
$W = \Z/2$, and 
$\mathrm{H}_\ret \Z = 0$ (the real spectrum being empty).
Moreover, 
all of the maps are algebra maps over the corresponding algebra for $\Z[1/\ell]$.
Thus the map for $\pi_{**} \mathrm{H}_\ret \Z$ is trivially an isomorphism, 
and the one for $\pi_{**} \1[\eta^{-1}]$ is an isomorphism because as an algebra it is determined by $W$ according to 
\cite[Proposition 5.2]{bachmann-etaZ}.
The isomorphism for $\mathrm{H}\F_\ell$ will hold if and only if $F(\tau) = \tau$, 
which holds provided $F(\tau^n) = \tau^n$ for some $n \ge 1$.
But,
for $n\gg 0$, 
$\tau^n$ exists over $\Z[1/\ell]$ 
(if $\ell=2$ this holds with $n=1$, and for $\ell$ odd see e.g., \cite[\S4.5(2)]{bachmann-bott}).

(2) Let $x, y \in \Spec(\Z[1/\ell])$ be the images of $\bar x, \bar y$.
We may assume $y$ is a specialization of $x$.
Let $X$ be the strict henselization of $\Spec(\Z[1/\ell])$ along $y$, 
with closed point $y'$.
By (1) applied with $X=X'$ we have $M(y') \wequi M(\bar x)$, 
and by applying it with $(X, \bar x, y) = (\{y'\}, \bar y, y')$ we get $M(y') \wequi M(\bar y)$.
\end{proof}

\begin{remark}
This common category $M(\bar x) \wequi M(\bar y)$ is known as $\ell$-complete $\mathrm{MU}$-based (even) synthetic spectra 
\cite{pstrkagowski2018synthetic}.
\end{remark}

\subsection{Criterion}
Recall that for $X \in \Schl$ the objects $\mathrm{H}\F_\ell, \mathrm{H}_\ret \Z \in \SH(X)$ are cellular 
and stable under base change.
For $\mathrm{H}\F_\ell$ this is \cite[Corollary 10.4, Theorem 8.22]{spitzweck2012commutative}.
For $\mathrm{H}_\ret \Z$ this follows from the expression $\mathrm{H}_\ret \Z \wequi o(\mathrm{H}\Z)[1/\rho]$ 
\cite{bachmann-real-etale}, 
where $o: \SH \to \SH(X)$ is the unique cocontinuous symmetric monoidal functor.
In particular, 
any morphism between $M(X)$ and $M(Y)$ in $\Premot_{\Z[1/\ell]}$ preserves these objects.

\begin{proposition} 
\label{prop:criterion}
Let $X_0, X_1, X_2, X_3 \in \Schl$ be essentially smooth over Dedekind schemes and consider a commutative square
\begin{equation}
\label{equation:diagramatstake}
\begin{CD}
M(X_3) @>>> M(X_1) \\
@VVV           @VVV  \\
M(X_2) @>>> M(X_0)
\end{CD}
\end{equation}
in $\Premot_{\Z[1/\ell]}$.
In order for \eqref{equation:diagramatstake} to be cocartesian, 
it suffices that the following conditions hold:
\begin{enumerate}
  \item For each $m$, the square
    \begin{equation*}
    \begin{CD}
      \map_{\SH(X_0)}(\Sigma^{0,*}\1, \mathrm{H}\F_\ell) @>>> \map_{\SH(X_1)}(\Sigma^{0,*}\1, \mathrm{H}\F_\ell) \\
          @VVV                                       @VVV                          \\
      \map_{\SH(X_2)}(\Sigma^{0,*}\1, \mathrm{H}\F_\ell) @>>> \map_{\SH(X_3)}(\Sigma^{0,*}\1, \mathrm{H}\F_\ell) \\
    \end{CD}
    \end{equation*}
    is cartesian.
  \item The square
    \begin{equation*}
    \begin{CD}
      \Gamma_\ret(X_0, \F_\ell') @>>> \Gamma_\ret(X_1, \F_\ell') \\
          @VVV                                       @VVV      \\
      \Gamma_\ret(X_2, \F_\ell') @>>> \Gamma_\ret(X_3, \F_\ell') \\
    \end{CD}
    \end{equation*}
    is cartesian.
    Here $\F_\ell'$ equals $\F_\ell$ if $\ell=\ell'$, and $\Z$ if $\ell' = \ell h$ (i.e., when $\ell$ even).
  \item If $2 \mid \ell$, then $\vcd_2(K(X_i)) < \infty$.
\end{enumerate}
If $X_0$ contains a primitive $\ell$-th root of unity, then condition (1) can be replaced by
\begin{enumerate}[(1')]
  \item For each $m$, the square
    \begin{equation*}
    \begin{CD}
      \Gamma_\Zar(X_0, R^m\epsilon_* \F_\ell) @>>> \Gamma_\Zar(X_1, R^m\epsilon_* \F_\ell) \\
          @VVV                                       @VVV                          \\
      \Gamma_\Zar(X_2, R^m\epsilon_* \F_\ell) @>>> \Gamma_\Zar(X_3, R^m\epsilon_* \F_\ell) \\
    \end{CD}
    \end{equation*}
    is cartesian.
\end{enumerate}
\end{proposition}
\begin{proof}
To conclude that the square is cocartesian it suffices, 
by Lemma \ref{lemm:premot-colim}, 
to prove the functor 
\[ 
\SH(X_0)_{\ell'}^{\comp\cell} 
\to \SH(X_1)_{\ell'}^{\comp\cell} \times_{\SH(X_3)_{\ell'}^{\comp\cell}} \SH(X_2)_{\ell'}^{\comp\cell} 
\] 
is fully faithful.
Let us denote by $p_{1*}: \SH(X_1)_{\ell'}^{\comp\cell} \to \SH(X_0)_{\ell'}^{\comp\cell}$ the right adjoint of the functor 
corresponding to $M(X_1) \to M(X_0)$, 
and similarly for $p_{2*}, p_{3*}$.
We need to prove that 
\[ 
\pi_{**}(\1_{\ell'}^\comp) 
\wequi 
\pi_{**}(p_{1*}(\1_{\ell'}^\comp) \times_{p_{3*}\1_{\ell'}^\comp}p_{2*}(\1_{\ell'}^\comp))
\]
Note that each of the left adjoints preserves the compact generators, 
which is true for any morphism in $\Premot$, 
and hence $p_{i*}$ preserves colimits and therefore it commutes with periodization.
Moreover, 
$p_{i*}$ commutes with $\wedge \mathcal{E}$ for every $\mathcal{E} \in \SH(X_0)^{\comp\cell}_{\ell'}$, 
and with completion at homotopy elements, by Lemma \ref{lemm:duality-yoga}.
We may check the desired equivalence after completing at $\eta$ and after inverting $\eta$, 
and similarly for other homotopy elements.
For the $\eta$-periodic statement, we further invert $2$ respectively complete at $2$.
In the $2$-complete (still $\eta$-periodic) case, 
either we have $2 \nmid \ell$ and the statement is vacuous, 
or $1/2 \in X_i$ and using the fundamental fiber sequence \cite[Corollary 1.2, Proposition 5.7]{bachmann-etaZ}, 
it suffices to establish the analogous equivalence for $\kw_{2,\ell'}^\comp$.
Recall that $\kw_{2,\ell'}^\comp$ is in fact cellular \cite[Proposition 5.7]{bachmann-etaZ}.
In the $2$-periodic (still $\eta$-periodic) case, 
arguing as in the proof of Proposition \ref{prop:rigidity}, 
it suffices to establish the analogous equivalence for $\mathrm{H}_\ret \F_\ell'$.
For the $\eta$-complete statement, 
arguing as in the proof of Proposition \ref{prop:rigidity}, 
we have $\1_{\eta,\ell'}^\comp \wequi \1_{\mathrm{H}\F_\ell}^\comp$ and we see that it suffices to establish 
the analogous equivalence for $\mathrm{H}\F_\ell$.
In summary, we need to prove the commutative square of ordinary spectra
\begin{equation*}
\begin{CD}
\map_{\SH(X_0)}(\Sigma^{0,*} \1, \mathcal{E}) @>>> \map_{\SH(X_1)}(\Sigma^{0,*} \1, \mathcal{E}) \\
     @VVV                                            @VVV \\
\map_{\SH(X_2)}(\Sigma^{0,*} \1, \mathcal{E}) @>>> \map_{\SH(X_3)}(\Sigma^{0,*} \1, \mathcal{E})
\end{CD}
\end{equation*}
is cartesian for all $* \in \Z$ and $\mathcal{E}$ one of $\kw_{2,\ell'}^\comp, \mathrm{H}\F_\ell, \mathrm{H}_\ret\F_\ell'$.

Before we start proving this, we need to make another preliminary remark.
Suppose that 
\begin{equation*}
\begin{CD}
\mathcal{E}^{(0)}_\bullet @>>> \mathcal{E}^{(1)}_\bullet \\
@VVV                    @VVV         \\
\mathcal{E}^{(2)}_\bullet @>>> \mathcal{E}^{(3)}_\bullet \\
\end{CD}
\end{equation*}
is a commutative diagram of filtered spectra such that $\mathcal{E}^{(i)}_n = 0$ for $n$ sufficiently small and 
the induced diagrams of associated graded objects
\begin{equation*}
\begin{CD}
gr_i \mathcal{E}^{(0)} @>>> gr_i \mathcal{E}^{(1)} \\
@VVV                    @VVV   \\
gr_i \mathcal{E}^{(2)} @>>> gr_i \mathcal{E}^{(3)} \\
\end{CD}
\end{equation*}
are pullbacks for each $i$.
Then the square 
\begin{equation*}
\begin{CD}
\lim_i \mathcal{E}^{(0)}_i @>>> \lim_i \mathcal{E}^{(1)}_i \\
@VVV                    @VVV         \\
\lim_i \mathcal{E}^{(2)}_i @>>> \lim_i \mathcal{E}^{(3)}_i \\
\end{CD}
\end{equation*}
is a pullback; indeed, an induction argument implies
\begin{equation*}
\begin{CD}
\mathcal{E}^{(0)}_i @>>> \mathcal{E}^{(1)}_i \\
@VVV            @VVV     \\
\mathcal{E}^{(2)}_i @>>> \mathcal{E}^{(3)}_i \\
\end{CD}
\end{equation*}
is a pullback for every $i$.

Next we show how the conditions (1)--(3) imply that the squares are cartesian.
The pullback square for $\mathrm{H}_\ret \F_\ell'$ is precisely condition (2), 
and the one for $\mathrm{H}\F_\ell$ is precisely condition (1).
The condition involving $\kw_{2,\ell'}^\comp$ is only non-vacuous if $2 \mid \ell$, 
whence $1/2 \in X_i$ and $\kw_2^\comp \wequi \kw_{2,\ell'}^\comp$.
Consider the filtration of $\kw$ by powers of $\beta$, pulled back to $X_i$.
The Postnikov filtration gives rise to the said filtration, 
and so it is complete, 
and $\mathrm{H}W$ gives all subquotients \cite[Theorem 4.4, Lemma 4.3]{bachmann-etaZ}.
Since $\kw$ is connective, 
the preliminary remark allows us to replace $\kw_2^\comp$ by $\mathrm{H}W_2^\comp$, 
which on mapping spectra yields $\Gamma(\ph, \ul{W}_2^\comp)$, 
where $\Gamma$ denotes global sections of a Nisnevich sheaf of spectra.
On mapping spectra the cellular motivic spectrum $\ul{K}^W$ \cite[Proposition 5.7, Theorem 4.4]{bachmann-etaZ} yields 
compatible filtrations of $\Gamma(\ph, \ul{W})$ by $\Gamma(\ph, \ul{I}^n)$, 
see \cite[Definition 2.6]{bachmann-etaZ}, 
where $I$ is the fundamental ideal of even dimensional quadratic forms.
Condition (3) together with \cite[Proposition 2.3]{bachmann-etaZ} implies $\lim_n \Gamma(\ph, \ul{I}^n/2) \wequi 0$. 
Thus the filtration $\Gamma(\ph, (\ul{W}/\ul{I}^n)_2^\comp)$ of $\Gamma(\ph, \ul{W}_2^\comp)$ is exhaustive.
Using the preliminary remark we may replace $\Gamma(\ph, \ul{W}_2^\comp)$ by $\Gamma(\ph, \ul{I}^*/\ul{I}^{*+1})$, 
which coincides with $\map(\Gmp{*}, (\mathrm{H}\Z/2)/\tau)$ according to \cite[Theorem 2.1, Lemma 2.7]{bachmann-etaZ}.
For this, 
we may establish the pullback square for $\mathrm{H}\F_\ell$,  
which implies the pullback square for $\mathrm{H}\Z_\ell^\comp$ and hence for 
$\mathrm{H}\Z_\ell^\comp/2 \wequi \mathrm{H}\Z/2$ since $2 \mid \ell$.

Finally, suppose that $\zeta_\ell \in X_0$.
This yields $\tau \in \pi_{0,-1}(\mathrm{H}\F_\ell)(X_0)$ given by the Bockstein on $[\zeta_\ell]$.
The cofibers of $\tau$-powers yields a filtration of $\mathrm{H}\F_\ell$ which pulls back to compatible filtrations 
on the $X_i$'s.
The explicit construction of the motivic complexes \cite[Theorem 3.9]{spitzweck2012commutative} shows that these 
filtrations are bounded, separated and exhaustive, and have subquotients $\Gamma_\Zar(X_i, R^m\epsilon_* \F_\ell)$.
Via the preliminary remark, the desired cartesian square thus reduces to condition (1').
\end{proof}

\begin{remark}\hfill
\begin{itemize}
\item In all our examples, 
the chain complexes in conditions (1') and (2) will be concentrated in a single degree.
\item If $\bar x$ is the spectrum of a separably closed field, 
then $\Gamma_\Zar(\bar x, R^m\epsilon_* \F_\ell) = 0$ for $m > 0$, 
and similarly $\Gamma_\ret(\bar x, \Z) = 0$.
\item
If the square \eqref{equation:diagramatstake} in Proposition \ref{prop:criterion} is cocartesian, 
then conditions (1) and (2) hold, and (1') holds whenever $\zeta_\ell \in X_0$.
Condition (3) is not necessary in general for the square to be cocartesian
(consider for example any square comprised of identity maps).
\end{itemize}
\end{remark}

\subsection{Models for stable motivic homotopy types}
\subsubsection{Arithmetic preliminaries}

\begin{lemma}
Suppose $K$ is a global field with ring of integers $\scr O_K$ and put $U = \Spec(\scr O_K[1/\ell])$.
If $\epsilon: U_\et \to U_\Zar$ is the change of topology functor, 
then 
\[ 
R^i \epsilon_* \mu_\ell 
\wequi 
\begin{cases} \mu_\ell & i=0 \\ 
a_\Zar \scr O^\times/\ell & i=1 \\ 
a_\ret \Z/(2,\ell) \oplus R & i=2 \\ 
a_\ret \Z/(2,\ell) & i > 2 
\end{cases}
\] 
The sheaf $R$ is determined by the exact sequence 
\[ 
0 \to R 
\to \bigoplus_{x \in \Spec(\scr O_K)^{(1)}} \F_\ell 
\to \F_\ell \oplus \bigoplus_{x \in U^{(1)}} i_{x*} \F_\ell \to 0 
\]
Here the middle term is a \emph{constant} sheaf whereas the right-hand term is a sum of a constant sheaf and skyscraper sheaves, 
and the map is given by addition in the first component and restriction in the others.
\end{lemma}
\begin{proof}
From \cite[Remark II.2.2]{milne2006arithmetic} we can read off the isomorphisms 
\[ 
R^0 \epsilon_* \Gm \wequi \Gm, 
\quad R^i \epsilon_* \Gm \wequi 0 \text{ for $i$ odd}, 
\quad R^i \epsilon_* \Gm \wequi a_\ret \R^\times \text{ for $i \ge 4$ even} 
\] 
and the short exact sequence 
\[ 
0 \to R^2 \epsilon_* \Gm 
\to a_\ret \Z/2 \oplus \bigoplus_{x \in \Spec(\scr O_K)^{(1)}} \Q/\Z 
\to \Q/\Z \oplus \bigoplus_{x \in U^{(1)}} i_{x*} \Q/\Z \to 0 
\]
For the exact sequence, 
recall $Br(K_v) = \Z/2$ if $v$ is a real place, 
$=0$ if $v$ is a complex place, 
and $=\Q/\Z$ is $v$ is a non-archimedean place \cite[p. 163, 193]{serre2013local}.
Moreover, 
the kernel of the restriction map is precisely the sum over the non-archimedean places missing in $U$.
The Kummer short exact sequence 
\[ 
0 \to \mu_\ell \to \Gm \xrightarrow{\ell} \Gm \to 0 
\] 
on $U_\et$ yields a long exact sequence for $R^i\epsilon_*$.
Since $R^i\epsilon_* \Gm$ vanishes in odd degrees, $R^i\epsilon_* \mu_\ell$ is given by the kernel or cokernel of 
multiplication by $\ell$.
This immediately yields the desired results for $i \ne 2,3$, 
and the snake lemma produces an exact sequence 
\[ 
0 \to R^2 \epsilon_* \mu_\ell 
\to a_\ret \Z/(2,\ell) \oplus \bigoplus_{x \in \Spec(\scr O_K)^{(1)}} \F_\ell 
\xrightarrow{b} \F_\ell \oplus \bigoplus_{x \in U^{(1)}} i_{x*} \F_\ell \to R^3 \epsilon_* \mu_\ell 
\to a_\ret \Z/(2,\ell) \to 0
\]
Since $b$ is a surjection of Zariski sheaves, the result follows.
\end{proof}

\begin{corollary} 
\label{cor:mot-coh-U}
Suppose $Pic(U)$ is uniquely $\ell$-divisible and $k(U)$ has a unique place of characteristic $\ell$.
Then $H^j(U, R^i \epsilon_* \mu_\ell) = 0$ for $j>0$ and 
\[ 
H^0(U, R^i\epsilon_* \mu_\ell) \wequi 
\begin{cases} \mu_\ell(U) & i=0 \\ 
\scr O^\times(U)/\ell & i=1 \\ 
(\Z/(2,\ell))^{\Sper(K)} & i>1 
\end{cases}
\]
\end{corollary}
\begin{proof}
Since $\mu_\ell|_{U_\Zar}$ is constant and constant sheaves are flasque, the claims for $i=0$ are clear.
The claims about $a_\ret \Z/(2,\ell)$ follow because $R(U) \wequi \Sper(k)$ is discrete.
Since the Zariski cohomological dimension of $U$ is $1$, 
it remains to show that $H^0_\Zar(U, \Gm/\ell) \wequi \scr O^\times(U)/\ell$, $H^1_\Zar(U, \Gm/\ell) = 0$, 
and $H^*_\Zar(U, R) = 0$ for $*=0, 1$.
Using the short exact sequences $0 \to \mu_\ell \to \Gm \to \ell\Gm \to 0$ and $0 \to \ell\Gm \to \Gm \to \Gm/\ell \to 0$, 
the first two claims are equivalent to unique $\ell$-divisibility of $Pic(U)$.
The exact sequence defining $R$ is, 
in fact, 
a flasque resolution, 
so its $H^0$ and $H^1$ are given by the kernel and cokernel of the induced map on global sections.
This induced map is an isomorphism as needed if and only if $\Spec(\scr O_K) \setminus U$ consists of precisely one point, 
which holds by assumption.
\end{proof}

\begin{lemma} \label{lemm:detect-powers}
Let $\ell$ be prime, $K$ a global field and $U \subset Spec(\scr O_K)$ open.
Let $H \subset \scr O^\times(U)/\ell$ be an arbitrary subgroup.
There exist $x_1, \dots, x_n \in U^{(1)}$ such that the restriction \[ H \subset \scr O^\times(U)/\ell \to \prod_i k(x)^\times/\ell \] is an isomorphism.
If $H$ is non-trivial, there exist infinitely many such choices.
\end{lemma}
\begin{proof}
First recall the following fact (see e.g.,  \cite[Exercise VI.1.2]{neukirch2013algebraic}): If $a \in \scr O(U)$ is an $\ell$-th power in $k(x)$ for all but finitely many $x \in U^{(1)}$, then $a$ is an $\ell$-th power.

If $H$ is non-trivial, pick $1 \ne a \in H$.
Since $H$ is a $\Z/\ell$-vector space, we may write $H = \lra{a} \times H'$, where $\lra{a} \wequi \Z/\ell$ is the subgroup generated by $a$.
By the above fact, there exists $x \in U^{(1)}$ such that the image of $a$ in $k(x)^\times/\ell$ is non-zero, and in fact infinitely many choices of $x$.
Since $k(x)$ is finite, $k(x)^\times/\ell \wequi \Z/\ell \wequi \lra{a}$.
We are thus reduced to proving the result for $H'$, and conclude by induction since $\scr O^\times(U)/\ell$ is finite according to Dirichlet's 
unit theorem \cite[Corollary 11.7]{neukirch2013algebraic}.
\end{proof}

In \cite{gras1986regular}, 
Gras introduced the narrow tame kernel $K_{2}^{+}(\scr O_{F})$ as the subgroup of $K_{2}(\scr O_{F})$ where the regular 
symbols on all the real embeddings of $F$ vanish, 
i.e., there is an exact sequence
\[
0\to K_{2}^{+}(\scr O_{F}) \to K_{2}(\scr O_{F}) \to \bigoplus^{r}\Z/2 \to 0
\]
We refer to \cite[Definition 7.8.1]{gras2003cft} for the arithmetic notion of $\ell$-regular number fields.
\begin{definition}
\label{definition:lregular}
Let $\ell$ be a prime number.
A number field $F$ is called $\ell$-regular if the $\ell$-Sylow subgroup of the narrow tame kernel 
$K_{2}^{+}(\scr O_{F})$ is trivial.
\end{definition}

See \cite{gras1986regular}, \cite{gras1989regular}, \cite{rognes2000regular}, \cite{berrick2011hermitian}
for complementary results about these families of number fields.
For example, the field of rational numbers $\Q$ is $\ell$-regular for every prime $\ell$, 
and $\Q(\zeta_\ell)$ is $\ell$-regular if $\ell$ is a regular prime number in the sense of Kummer \cite{Was-cyclotomic}.
%e.g., $\ell=3, 5, 7, 11, 13, 17, 19, 23, 29, 31, 41, 43, 47, 53, 61, 71, 73, 79, 83, 89, 97$, 
%$107, 109, 113, 127, 137, 139, 151, 163, 167, 173, 179, 181, 191$, etc.
In \cite{siegel1964}, Siegel conjectured there are infinitely many regular prime numbers. 

We have the following explicit characterization of $\ell$-regular number fields.
\begin{proposition}
\label{proposition:characterizationofregularnumberfields}
Let $F$ be a number field.
We write ${\scr O}_{F}^{\prime}$ for the ring of $\ell$-integers $\scr O_F[1/\ell]$.
\begin{enumerate}
\item $F$ is $2$-regular if and only if the prime ideal $(2)$ does not split in $F/\Q$ and the narrow Picard group 
$Pic_{+}({\scr O}_{F}^{\prime})$ has odd order.
\item Let $\ell$ be an odd prime number and assume $\mu_{\ell}\subset F$.
Then $F$ is $\ell$-regular if and only if the prime ideal $(\ell)$ does not split in $F/\Q$ and the 
$\ell$-Sylow subgroup of the Picard group $Pic({\scr O}_{F}^{\prime})$ is trivial.
\item Let $\ell$ be an odd prime number.  
Assume $\mu_{\ell}\not\subset F$ and $F$ contains the maximal real subfield of $\Q(\zeta_{\ell})$.
Then $F$ is $\ell$-regular if and only if the prime ideals above $(\ell)$ in $F$ do not split in the 
quadratic extension $F(\zeta_{\ell})/F$ and the $\ell$-Sylow subgroups of the Picard groups $Pic({\scr O}_{F})$ 
and $Pic({\scr O}_{F(\zeta_{\ell})})$ are isomorphic.
\end{enumerate}
\end{proposition}
\begin{proof}
This is a reformulation of \cite[Corollary on pp. 328-329]{gras1986regular}.
See also \cite[Proposition 2.2]{rognes2000regular} when $\ell=2$.
\end{proof}

For further reference we recall that a commutative square of abelian groups
\begin{equation*}
\begin{CD}
A_1 @>i_1>> A_2 \\
@Vi_2VV   @VVp_1V  \\
A_3 @>p_2>> A_4
\end{CD}
\end{equation*}
is called \emph{bicartesian} if it is a pullback when viewed as a commutative square of spectra.

\begin{comment}
\begin{lemma}
The following are equivalent.
\begin{enumerate}
\item The square is bicartesian.
\item The square is both a pullback and a pushout in abelian groups.
\item The associated sequence 
\[ 0 \to A_1 \xrightarrow{(i_1,i_2)} A_2 \oplus A_2 \xrightarrow{p_1-p_2} A_4 \to 0 \] 
is exact.
\item The square is pullback and $p_1 + p_2$ is surjective.
\end{enumerate}
\end{lemma}
\begin{proof}
(1) is equivalent to (3) by the relationship between pullback squares and fiber sequences in additive categories \cite{TODO}.
(1) implies (2) since pullback squares are pushout squares in stable categories, and (2) implies (4) by the relationship 
between pushouts and cofibers in additive categories.
(4) implies (3) since the sequence is left exact for any pullback square, 
and the last map is surjective if and only if $p_1+p_2$ is.
\end{proof}
\end{comment}

\subsubsection{Stable motivic homotopy types of $2$-regular number fields}

\begin{theorem} 
\label{thm:main-1}
Suppose $F$ is a $2$-regular number field with $r$ real and $c$ pairs of complex embeddings.
Let $x, y_1, \dots, y_{c} \in \Spec({\scr O}_{F}^{\prime})$ be closed points.
\begin{enumerate}
\item There is a canonical commutative square in $\Premot$
\begin{equation*}
\begin{CD}
M_{2h}(\C^{c+r}) @>>> M_{2h}(\R^{r}) \amalg \coprod_i M_{2h}(y_i) \\
@VVV       @VVV \\
M_{2h}(x) @>>> M_{2h}({\scr O}_{F}^{\prime})
\end{CD}
\end{equation*}
\item The square in (1) is a pushout if and only if there is a naturally induced isomorphism
\begin{equation*}
({\scr O}_{F}^{\prime})^\times/2 \wequi 
(\R^\times/2)^{r} \times k(x)^\times/2 \times \prod_i k(y_i)^\times/2 \quad (\wequi (\Z/2)^{1+r+c})
\end{equation*}
\item There exist infinitely many choices of $x, y_1, \dots, y_{c}$ such that the map in (2) is an isomorphism.
\end{enumerate}
\end{theorem}
\begin{proof}
To simplify notation, throughout this proof we put $M := M_{2h}$.

(1) For $z \in \Spec({\scr O}_{F}^{\prime})$ and $\alpha: K \hookrightarrow \C$, 
Lemma \ref{lemm:synpt-basics} furnishes a map 
$f_{z,\alpha}: M(\C) \to M(z)$ and a homotopy between 
$M(\C) \to M(z) \to M({\scr O}_{F}^{\prime})$ and the map $M(\C) \to M({\scr O}_{F}^{\prime})$ induced by $\alpha$.
In (1), 
the bottom and right-hand maps are the canonical ones.
Write $\alpha_1, \bar \alpha_1, \dots, \alpha_c, \bar \alpha_c, \beta_1, \dots, \beta_r$ for the complex and real embeddings.
Let $\alpha_{c+i} = \iota \circ \beta_i$, where $\iota: \R \to \C$ is the canonical embedding.
The left-hand map is $f_{x,\alpha_i}$ on component $i$.
The top map is $f_{y_i, \alpha_i}$ on the $i$-th component if $i \le c$, and induced by $\iota$ on the remaining components.
In all cases, 
the induced composite map $M(\C) \to M({\scr O}_{F}^{\prime})$ is either equal or homotopic to the map induced by $\alpha_i$.
Thus the square commutes.

(2)
We use the criterion from Proposition \ref{prop:criterion}.
Condition (3) holds since the fields are finitely generated.
For condition (2), the square
\begin{equation*}
\begin{CD}
(\Z/2)^{r} @>>> 0 \\
@VVV           @VVV \\
(\Z/2)^{r} @>>> 0
\end{CD}
\end{equation*}
is clearly bicartesian because the map 
${\scr O}_{F}^{\prime} \to \R^r$ induces an isomorphism on real spectra.

Next we check condition (1').
Owing to \cite[Proposition 2.1(5)]{berrick2011hermitian} the $2$-regularity assumption implies 
$Pic({\scr O}_{F}^{\prime})$ has odd order, 
so it is uniquely $2$-divisible, 
and $F$ has only one place of characteristic $2$. 
Thus Corollary \ref{cor:mot-coh-U} applies and it remains to check the bicartesianess of three squares.
The first one is
\begin{equation*}
\begin{CD}
\Z/2 @>>> (\Z/2)^{r+c} \\
@VVV        @VVV           \\
\Z/2 @>>> (\Z/2)^{r+c}
\end{CD}
\end{equation*}
Observe that if $X, Y$ are connected schemes and $f: M(X) \to M(Y)$ is any map in $\Premot_{\Z[1/2]}$, 
then $f^*: H^0(Y, \Z/2) \to H^0(X, \Z/2)$ is an isomorphism.
Indeed this reduces to the case of the structure map $M(X) \to M(\Z[1/2])$, where it is obvious.
Thus the square for $m=0$ is bicartesian because the vertical maps are isomorphisms.
When $m=2$ the square is the same as in condition (2) above, and hence it is bicartesian.
The remaining square for $m=1$ takes the form
\begin{equation*}
\begin{CD}
({\scr O}_{F}^{\prime})^\times/2 @>>> k(x)^\times/2 \\
@VVV           @VVV    \\
(\R^\times/2)^{r} \times \prod_i k(y_i)^\times/2 @>>> (\C^\times/2)^{r+c} = 0
\end{CD}
\end{equation*}
Since the inclusion of abelian groups into spectra preserves finite products, 
this square is bicartesian if and only if the stated condition holds.

(3)
Dirichlet's unit theorem \cite[Corollary 11.7]{neukirch2013algebraic} implies 
$({\scr O}_{F}^{\prime})^\times \wequi \mu({\scr O}_{F}^{\prime}) \times \Z^{r+c}$; 
here $\mu({\scr O}_{F}^{\prime})$ is the finite abelian group of roots of unity in ${\scr O}_{F}^{\prime}$.
It is cyclic, being a finite multiplicative subgroup of a field,  
and since $\{\pm 1\} \in \mu({\scr O}_{F}^{\prime})$ the group has even order.
It follows that $({\scr O}_{F}^{\prime})^\times/2 \wequi \Z/2 \times (\Z/2)^{r+c}$.
Moreover, 
$2$-regularity implies the naturally induced map 
$({\scr O}_{F}^{\prime})^\times/2 \to (\R^\times/2)^{r} \wequi (\Z/2)^{r}$ is surjective 
\cite[Proposition 2.1(5)]{berrick2011hermitian}; we write $U_+ \wequi (\Z/2)^{1+c}$ for its kernel.
The condition in part (2) holds if and only if the induced map $U_+ \to k(x)^\times/2 \times \prod_i k(y_i)^\times/2$ 
is an isomorphism.
Lemma \ref{lemm:detect-powers} implies the latter is true for infinitely many choices of $x, y_i$.
\end{proof}

\begin{remark}
As in Examples \ref{ex:eta-periodic-h} and \ref{ex:n-compl-strong}, 
Theorem \ref{thm:main-1}(1) implies similar pushout squares with respect to completions at $2$, $h$, 
and with respect to periodizations at $\rho$, $\eta$.
For example, we have a pushout square in $\Premot$
\begin{equation*}
\begin{CD}
M(\C^{c+r})[\eta^{-1}] @>>> M(\R^{r})[\eta^{-1}] \amalg \coprod_i M(y_i)[\eta^{-1}] \\
@VVV       @VVV \\
M(x)[\eta^{-1}] @>>> M({\scr O}_{F}^{\prime})[\eta^{-1}]
\end{CD}
\end{equation*}
\end{remark}

\begin{remark} \label{rmk:wedge}
The various embeddings $\alpha_i: K \to \C$ differ by automorphisms of $\C$.
It follows that one may choose the maps $f_{x,\alpha_i}$ to be of the form $\sigma_i \circ f_{x,\alpha_1}$.
Thus, applying an automorphism of $M_{2h}(\C^{r+s})$ in the square of Theorem \ref{thm:main-1}, we may assume that all the left hand vertical maps $M_{2h}(\C) \to M_{2h}(x)$ are the same.
The square being a pushout now is equivalent to saying that there are lifts of  $M_{2h}(x), M_{2h}(y_i)$ to $\Premot_{M_{2h}(\C)/}$, 
and an equivalence
\[ 
M_{2h}({\scr O}_{F}^{\prime})
\wequi 
\bigvee^r M_{2h}(\R) \vee M_{2h}(x) \vee \bigvee^c M_{2h}(y_i) 
\] 
Here $\vee$ denotes the coproduct in $\Premot_{M(\C)/}$.
\end{remark}

\begin{example} 
\label{ex:Z12}
When $F=\Q$ we consider $\Z[1/2]^\times \wequi \{\pm 1\} \times \{(1/2)^n\}$ and $\Z[1/2]^\times/2 \wequi \Z/2\{-1, 2\}$.
Here $\Z/2\{2\}$ is the kernel of the surjection $\Z[1/2]^\times/2\to\R^\times/2$. 
We need to find a closed point $x \in \Spec(\Z[1/2])$ such that $2$ is not a square in $k(x)$.
This holds when $k(x) = \Spec(\F_q)$, where $q\equiv \pm 3\bmod 8$.
In particular the canonical map \[ \SH(\Z[1/2])_2^{\comp\cell} \to \SH(\R)_2^{\comp\cell} \times_{\SH(\C)_2^{\comp\cell}} \SH(\F_3)_2^{\comp\cell}  \]
is fully faithful.
To deduce Theorem \ref{theorem:mainthmintro} from the introduction, let $\scr E \in \SH(\Z[1/2])_2^{\comp\cell}$ and compute $\map(\1, \scr E)$ using the above square.
\end{example}

\subsubsection{Stable motivic homotopy types of $\ell$-regular number fields}

\begin{theorem} 
\label{thm:main-2}
%Suppose $\ell$ is an odd prime, 
%$F$ contains a primitive $\ell$-th root of unity, 
%$Pic({\scr O}_{F}^{\prime})$ is uniquely $\ell$-divisible 
%\todo{shown using localization for $Pic$?}
%(e.g., $Pic(\scr O_K)$ uniquely $\ell$-divisible\NB{really?}) 
%and $F$ has precisely one prime lying above $\ell$.
Let $F$ be a number field with $c$ pairs of complex embeddings and $\ell$ be an odd prime number.
Suppose $F$ is $\ell$-regular and $\mu_{\ell}\subset F$. 
Let $x, y_1, \dots, y_c \in \Spec({\scr O}_{F}^{\prime})$ be closed points.
\begin{enumerate}
\item There is a canonical commutative square in $\Premot$
\begin{equation*}
\begin{CD}
M_\ell(\C^{c}) @>>> \coprod^{c} M_\ell(y_i) \\
@VVV                   @VVV  \\
M_\ell(x)    @>>> M_\ell({\scr O}_{F}^{\prime})
\end{CD}
\end{equation*}
\item The square is a pushout if and only if there is a naturally induced isomorphism
\[ 
({\scr O}_{F}^{\prime})^\times/\ell 
\wequi k(x)^\times/\ell \times \prod_{i} k(y_i)^\times/\ell 
\quad (\wequi (\F_\ell)^{1+c}) \]
\item There exist infinitely many choices of $x,y_1, \dots, y_c$ satisfying (2).
\end{enumerate}
\end{theorem}
\begin{proof}
The proof is essentially the same as that of Theorem \ref{thm:main-1}.
The maps $x, y_i\to \Spec({\scr O}_{F}^{\prime})$ together with choices of embeddings of $K$ into $\C$ induce, 
via Lemma \ref{lemm:synpt-basics}, 
the maps $M_\ell(\C) \to M_\ell(x),M_\ell(y_i)$ in the commutative square.
One verifies, 
using Corollary \ref{cor:mot-coh-U} and $\Z/(2, \ell) = 0$, 
that condition (1') of Proposition \ref{prop:criterion} reduces to the condition stated in (2). 
The other conditions hold trivially; 
since $K$ contains a primitive $\ell$-th root of unity, 
the real spectrum $\Sper({\scr O}_{F}^{\prime}) = \emptyset$.
The existence of infinitely many choices in (3) follows from Lemma \ref{lemm:detect-powers}.
\end{proof}

\begin{remark}
Arguing as in Remark \ref{rmk:wedge}, we find that there are lifts of 
$M_{\ell}(x), M_{\ell}(y_i)$ to $\Premot_{M_{\ell}(\C)/}$, 
and an equivalence
\[ 
M_{\ell}({\scr O}_{F}^{\prime})
\wequi 
\bigvee^c M_{\ell}(y_i) \vee M_{\ell}(x)
\]
\end{remark}

\begin{example}
Theorem \ref{thm:main-2} applies to $F=\Q(\zeta_\ell)$ if $\ell$ is regular ---
we note that $(\ell)$ is totally ramified in $F$ 
%by \cite[Lemma 10.1]{neukirch2013algebraic} 
and $K_{2}(\Z[\zeta_\ell])/\ell\equiv \mu_{\ell}\otimes Pic(\Z[\zeta_\ell])$. 
In this case, 
${\scr O}_{F}^{\prime}=\Z[1/\ell,\zeta_\ell]$ and $k(x)=\F_{p}$, 
where $p$ is a prime number which is congruent to $1 \bmod \ell$ but is not congruent to $1 \bmod \ell^{2}$ by 
\cite[Example 1.9]{dwyer1994topological}.
%After Lemma 4.4 they do $K=\Q(\mu_\ell)$ and find a wedge of $(\ell+1)/2$ circles. 
%Since $c=(\ell-1)/2$, this is indeed the same as $c+1$ circles that we find.
\end{example}

\begin{theorem} 
\label{thm:main-3}
Let $\ell$ an odd regular prime and $p\ne \ell$ a prime number.
There is a commutative square in $\Premot_{\Z[1/\ell]}$
\begin{equation*}
\begin{CD}
M_\ell(\C) @>>> M_\ell(\R) \\
@VVV              @VVV     \\
M_\ell(\F_p) @>>> M_\ell(\Z[1/\ell])
\end{CD}
\end{equation*}
The square is a pushout if $p$ generates the multiplicative group of units $(\Z/{\ell^2})^\times$.
\end{theorem}
\begin{proof}
We get the square from Lemma \ref{lemm:synpt-basics} and proceed by verifying the conditions in 
Proposition \ref{prop:criterion}.
Since $\Z[1/\ell]$ has a unique real embedding, condition (2) holds.
Condition (3) is vacuous.
Next we verify condition (1).
Let us write $\Gamma(X, \F_\ell(i))$ for the motivic complex and 
$\Gamma_\et(X, \F_\ell(i)) \wequi \Gamma_\et(X, \mu_\ell^{\otimes i})$ for its étale version.
If $A = \Z[1/\ell, \zeta_\ell]$, 
then $H^0_\et(A, \F_\ell) = \F_\ell$, $H^1_\et(A, \F_\ell) = A^\times/\ell$ and $H^*_\et(A, \F_\ell) = 0$ else, 
see \cite[Remark II.2.2]{milne2006arithmetic}.  
Corollary \ref{cor:mot-coh-U} implies that $\Gamma(A, \F_\ell(i)) \wequi \Gamma_\et(A, \F_\ell(i))_{\ge -i}$.
A transfer argument shows $\Gamma(\Z[1/\ell], \F_\ell(i))$ is a summand of $\Gamma(A, \F_\ell(i))$, 
and similarly for $\Gamma_\et$. 
We deduce the equivalence
\[ 
\Gamma(\Z[1/\ell], \F_\ell(i)) \wequi \Gamma_\et(\Z[1/\ell], \F_\ell(i))_{\ge -i}
\]
The same is true for $\C,\R,\F_p$ since they are Nisnevich local.
Consequently condition (1) will hold if the square
\begin{equation*}
\begin{CD}
\Gamma_\et(\Z[1/\ell], \F_\ell(i)) @>>> \Gamma_\et(\F_p, \F_\ell(i)) \\
@VVV                                        @VVV \\
\Gamma_\et(\R, \F_\ell(i)) @>>> \Gamma_\et(\C, \F_\ell(i))
\end{CD}
\end{equation*}
is cartesian and the maps 
\[ 
H^i_\et(\R, \F_\ell(i)) \oplus H^i_\et(\F_p, \F_\ell(i)) \to H^i_\et(\C, \F_\ell(i)) 
\] 
are surjective for every $i$.
The first condition holds by \cite[Theorem 2.1]{dwyer1994topological}.
The second condition is vacuous when $i>0$ and easily verified for $i=0$.
\end{proof}

\begin{remark}
By adjoining an $\ell$-th root of unity, one obtains the commutative square
\begin{equation*}
\begin{CD}
M_\ell(\C \otimes \Z[\zeta_\ell]) @>>> M_\ell(\R \otimes \Z[\zeta_\ell]) \\
@VVV              @VVV     \\
M_\ell(\F_p \otimes \Z[\zeta_\ell]) @>>> M_\ell(\Z[1/\ell, \zeta_\ell])
\end{CD}
\end{equation*}
This induces a cartesian square in étale cohomology, 
but \emph{not} in motivic cohomology
(since e.g., the group $H^{1,0}(\Z[1/\ell,\zeta_\ell],\F_\ell)=0$ but the corresponding map on $H^{0,0}$ is not surjective).
\end{remark}

\subsubsection{Relation to étale homotopy types}
Corresponding to the squares in Theorems \ref{thm:main-1}, \ref{thm:main-2} and \ref{thm:main-3}, 
there are analogous squares of \emph{étale homotopy types}; 
in fact, 
Lemma \ref{lemm:synpt-basics}, 
the only non-formal input in the construction of the said squares, 
also holds for étale homotopy types.
Due to the equivalence
\[
\mathrm{H}_\et\Z/{\ell^n} \wequi \mathrm{H}\Z/{\ell^n}[(\tau)^{-1}] \in \SH(S)_\ell^{\comp\cell}
\]
from e.g., \cite[Theorem 7.4]{bachmann-bott}, 
our cocartesian squares in $\Premot_{\Z[1/\ell]}$ induce cartesian squares in étale cohomology with 
$\Z/{\ell^n}(i)$-coefficients.
By Dwyer--Friedlander \cite[Theorem 2.1]{dwyer1994topological} \cite[pp. 144--145]{dwyer1983conjectural}, 
the resulting squares of étale homotopy types become pushouts after appropriate homological localization.

By analyzing the proof of Theorem \ref{thm:main-3}, 
one sees that condition (1) in Proposition \ref{prop:criterion} is satisfied if the following hold:
\begin{itemize}
\item $\dim X_0 \le 1$, $\dim X_n = 0$ else.
\item The induced square of étale cohomology with $\F_\ell(i)$-coefficients is cartesian.
\item The induced square of Zariski cohomology with $\F_\ell$-coefficients is cartesian.
\item The group $H^1_\Zar(X_0, R^i \epsilon_* \F_\ell(i)) = 0$ for $i \ge 0$.
\end{itemize}

\section{Applications to slice completeness and universal motivic invariants} 
\label{sec:applications}

We apply the results in Section \ref{sec:models} to show slice completeness and compute the endomorphism ring 
of the motivic sphere over regular number rings.
Our completeness result for Voevodsky's slice filtration \cite{voevodsky-slice-filtration} is motivated by applications 
such as motivic generalizations of Thomason's étale descent theorem for algebraic $K$-theory in \cite{elmanto2017-bott}
and \cite{bachmann-bott},
convergence of the slice filtration \cite{levine2013convergence},
the solution of Milnor's conjecture on quadratic forms in \cite{rondigs2016milnorconjecture},
computations of universal motivic invariants in \cite{rondigs2016first} and of hermitian $K$-groups in \cite{kylling2018hermitian}.

For the standard nomenclature associated with the slice filtration, 
such as the effective covers $f_{q}$ and the effective cocovers $f^{q}$, 
and the slice completion $\scomp$ we refer to \cite[\S 3, (3.1), (3.3), (3.10)]{rso-solves}.
Let $\SH(S)_{\ge 0}$ denote the connective motivic spectra with respect to the homotopy $t$-structure on $\SH(S)$
\cite[\S2.1]{hoyois-algebraic-cobordism}.
The notion of a cell presentation of finite type is defined in \cite[\S3.3]{rondigs2016first}.

\begin{proposition} \label{prop:slice-app}
Suppose $F$ is a $2$-regular number field and set ${\scr O}_{F}^{\prime}:= \scr O_F[1/2]$. 
\begin{enumerate}
\item Let $\mathcal{E}_\bullet \in \SH({\scr O}_{F}^{\prime})_2^{\comp\cell}$ be a tower such that 
$\lim_n p_i^*(\mathcal{E}_n) \wequi 0$, 
where $p_i^*$ denotes the pullback to any of the fields in Theorem \ref{thm:main-1}(1).
Then $\lim_n \mathcal{E}_n \wequi 0$ is contractible.
\item If $\mathcal{E} \in \SH({\scr O}_{F}^{\prime})^\veff \cap \SH({\scr O}_{F}^{\prime})^\cell$ is cellular and very effective,
then $\mathcal{E}/2$ is $\eta$-complete on homotopy.
\item Let $\mathcal{E} \in \SH({\scr O}_{F}^{\prime})_{\ge 0} \cap \SH({\scr O}_{F}^{\prime})^\cell$ 
and assume the slices of $\mathcal{E}$ are cellular and stable under base change.
Then there is an isomorphism 
\[
\pi_{\ast,\ast}(\lim_n f^n(\mathcal{E})/(2,\rho)) \wequi \pi_{\ast,\ast}(\mathcal{E}/(2,\rho))
\]
\item Let $\mathcal{E} \in\SH({\scr O}_{F}^{\prime})^\eff\cap\SH({\scr O}_{F}^{\prime})^\cell$ be cellular and effective. 
Assume $\mathcal{E}/2$ has a $\Z_{(2)}$-cell presentation of finite type and its slices are cellular and stable under base change.
Then $\mathcal{E}/(2,\eta)$ is slice complete on homotopy and 
\[
\pi_{\ast,\ast}(\scomp(\mathcal{E})_2^\comp) \wequi \pi_{\ast,\ast}( \mathcal{E}_{2,\eta}^\comp)
\]
\end{enumerate}
In particular, there is an isomorphism
\[
\pi_{\ast,\ast}(\scomp(\1)_2^\comp)\wequi \pi_{\ast,\ast}(\1_2^\comp)
\]
\end{proposition}
\begin{proof}
(1) Let $I = \{* \to * \leftarrow *\}$ be the category so that $\lim_I$ means pullback.
For all $X \in \SH({\scr O}_{F}^{\prime})_2^{\comp\cell}$ we compute 
\begin{gather*} 
\Map(X, \lim_n \mathcal{E}_n) \wequi \lim_n \Map(X, \mathcal{E}_n) \wequi 
\lim_n \lim_{i \in I} \Map(p_i^*(X), p_i^*(\mathcal{E}_n)) \\ 
\wequi \lim_i \lim_n \Map(p_i^*(X), p_i^*(\mathcal{E}_n)) \wequi 
\lim_i \Map(p_i(X), \lim_n p_i^*(\mathcal{E}_n)) \wequi 0
\end{gather*}
The result follows.

(2) Recall that $\mathcal{E}$ is $\eta$-complete if and only if 
\[ 
\lim \left[ \cdots \xrightarrow{\eta} \Sigma^{2,2} \mathcal{E} 
\xrightarrow{\eta} \Sigma^{1,1} \mathcal{E} \xrightarrow{\eta} \mathcal{E} \right]
\wequi 0
\]
Thus by (1) it suffices to check $p_i^*(\mathcal{E}/2)$ is $\eta$-complete for each $i$, 
which holds by \cite[Theorem 5.1]{bachmann-eta}.

(3) The claim holds if and only if $\lim_n f_n(\mathcal{E})/(2,\rho) \wequi 0$ on homotopy groups, 
or equivalently when computed in $\SH({\scr O}_{F}^{\prime})^\cell$.
The assumptions imply $f_n(\mathcal{E}) \in \SH({\scr O}_{F}^{\prime})^\cell$ and 
$p_i^* f_n \mathcal{E} \wequi f_n p_i^* \mathcal{E}$.
Hence by (1) it suffices to note that $\lim_n f^n(p_i^* \mathcal{E})/(2,\rho) \wequi p_i^*(\mathcal{E})/(2,\rho)$
owing to \cite[Proposition 5.2]{bachmann-eta}.

(4) For the first statement we need to prove $\lim_n f_n(\mathcal{E}/(2,\eta)) \wequi 0$ on homotopy groups.
As in (3), 
this reduces to the same statement over fields, 
which holds by \cite[Proposition 3.49]{rondigs2016first}.
For the second statement we need to show $\scomp(\mathcal{E}/2) \wequi \mathcal{E}_\eta^\comp/2$, 
which holds by the proof of \cite[Lemma 3.13]{rondigs2016first}: 
$\scomp(\mathcal{E}/2)$ is $\eta$-complete since $\mathcal{E}/2$ is effective, 
and $\scomp(\mathcal{E}/2)/\eta \wequi \scomp(\mathcal{E}/(2,\eta)) \wequi \mathcal{E}/(2,\eta)$ 
--- the first equivalence holds by inspection of the slices.
The final statement follows since the slices of $\1_{(2)}$ over ${\scr O}_{F}^{\prime}$ are known and have the 
desired properties by \cite[Remark 2.2, Theorem 2.12]{rondigs2016first}.
\end{proof}
\todo{slice completeness for $\ell$ odd?}
\NB{$E \to L_b E$ is a $\rho$-periodic equivalence, 
whence its fiber is determined by the $\rho$-completion: 
$F \wequi \fib(F_\rho^\comp \to F_\rho^\comp[\rho^{-1}])$. 
In particular, 
$F$ is concentrated in Chow $\le n$ if the same holds for $F_\rho^\comp$.}

\begin{remark}
We expect that analogs of Proposition \ref{prop:slice-app} hold over more general base schemes.
Moreover, 
we expect that these results hold without the qualification ``on homotopy.''
Both shortcomings are a result of our specific technique for accessing global sections of cellular spectra over arithmetic
base schemes.
\end{remark}

Recall that any unit $a \in \scr O(S)^\times$ gives rise to a map $[a]: \1 \to S^{1,1} \in \SH(S)$, 
and hence an element
\[ 
\lra{a} := 1 + \eta[a] \in \pi_{0,0}(\1_S) 
\]
This turns $\pi_{0,0}(\1)$ into an $\Z[\scr O(S)^\times]$-algebra.
We made use of the algebra structure in the formulation of Theorem \ref{theorem:2mainthmintro} for $\Z[1/2]$.
The generalization to $2$-regular number rings takes the following form.

\begin{theorem}
\label{theorem:general2mainthmintro}
Suppose $F$ is a $2$-regular number field with $r$ real embeddings and $c$ pairs of complex embeddings.
For the endomorphism ring of the motivic sphere over the base scheme ${\scr O}_{F}^{\prime}:= \scr O_F[1/2]$ 
there is an isomorphism of $\Z[({\scr O}_{F}^{\prime})^\times]$-algebras
\[ 
\pi_{0,0}(\1_{{\scr O}_{F}^{\prime}}) \otimes \Z_{(2)} 
\wequi 
\GW({\scr O}_{F}^{\prime}) \otimes \Z_{(2)}
\]
induced by the unit map $\1 \to \KO$.
Moreover, we have the vanishing result
\[ 
\pi_{*,0}(\1_{\scr O_F'}) \otimes \Z_{(2)} 
= 
0 
\quad\text{for}\quad *<0
\]
\end{theorem}
\begin{proof}
The presentation of Grothendieck-Witt rings of fields of characteristic $\neq 2$ by generators and relations given 
in \cite[Theorem 4.1]{lam-quadratic-forms} implies there are $\Z[({\scr O}_{F}^{\prime})^\times]$-algebra isomorphisms 
\[
\GW(\R) \wequi \Z \oplus \Z\{\lra{-1}\}, 
\GW(\C) \wequi \Z, 
%\GW(\F_3) \wequi \Z \oplus \Z/2\{\lra{2}-1\}
\GW(\F_q) \wequi \Z \oplus \Z/2
\]
In the isomorphism for $\GW(\F_q)$, 
the right-hand side has trivial multiplication on the square class group $\F_q^{\times}/(\F_q^{\times})^{2}\wequi \Z/2$.
As such, 
every $n$-dimensional form in $\GW(\F_q)$ can be written as either 
$n\langle 1\rangle$ or $(n-1)\langle 1\rangle\oplus \langle a\rangle$, 
where $a$ is a non-square element in $\F_q^{\times}$ (we may choose $a=-1$ if and only if $q\equiv 3\bmod 4$).
Moreover, 
by \cite[Proposition 2.1(7)]{berrick2011hermitian} and the proof of \cite[Theorem 5.8]{bachmann-euler} one deduces 
the $\Z[({\scr O}_{F}^{\prime})^\times]$-algebra isomorphism
%\[
%\GW(\Z[1/2]) \wequi \Z \oplus \Z\{\lra{-1}\} \oplus \Z/2\{\lra{2}-1\}
%\]
\[
\GW({\scr O}_{F}^{\prime}) 
\wequi 
\Z^{1+r} \oplus (\Z/2)^{1+c}
\]
Thus, 
for the closed points $x, y_1, \dots, y_{c} \in \Spec({\scr O}_{F}^{\prime})$ in the notation of Theorem \ref{thm:main-1}, 
there is a pullback square of $\Z[({\scr O}_{F}^{\prime})^\times]$-algebras
\begin{equation}
\label{equation:GWsquare}
\begin{CD}
\GW({\scr O}_{F}^{\prime}) @>>> \bigoplus^{r} \GW(\R) \oplus \bigoplus^{c} \GW(y_i) \\
@VVV              @VVV    \\
\GW(x)   @>>> \bigoplus^{r+c} \GW(\C)
\end{CD}
\end{equation}
%$\Z[\Z[1/2]^\times]$-algebras
%\begin{equation}
%\label{equation:GWsquare}
%\begin{CD}
%\GW(\Z[1/2]) @>>> \GW(\R) \\
%@VVV              @VVV    \\
%\GW(\F_3)   @>>> \GW(\C)
%\end{CD}
%\end{equation}
The Grothendieck-Witt rings appearing in \eqref{equation:GWsquare} are quotients of $\Z[({\scr O}_{F}^{\prime})^\times]$.
Thus the maps in \eqref{equation:GWsquare} are unique as $\Z[({\scr O}_{F}^{\prime})^\times]$-algebra maps.
Since $2$-adic completion is exact on finitely generated abelian groups, 
this square remains cartesian after $2$-adic completion.

Consider the long exact sequence of homotopy groups associated with the pullback square
\begin{equation}
\label{equation:2completedsquare1}
\begin{CD}
\map(\1_{{\scr O}_{F}^{\prime}},\1_2^\comp) @>>> \map(\bigoplus^{r}\1_{\R} \oplus \bigoplus^{c} \1_{y_{i}}, \1_2^\comp) \\
@VVV                                  @VVaV \\
\map(\1_{x}, \1_2^\comp) @>>> \map(\bigoplus^{r+c}\1_{\C}, \1_2^\comp) \\
\end{CD}
\end{equation}
%\begin{equation}
%\label{equation:2completedsquare1}
%\begin{CD}
%\map(\1_{\Z[1/2]}, \1_2^\comp) @>>> \map(\1_{\R}, \1_2^\comp) \\
%@VVV                                  @VVaV \\
%\map(\1_{\F_3}, \1_2^\comp) @>>> \map(\1_{\C}, \1_2^\comp) \\
%\end{CD}
%\end{equation}
We have $\pi_{\ast,0}(\1_2^\comp)(\C) \wequi (\pi_*^s)_2^\comp$ by \cite[Corollary 2]{levine2014comparison}. 
It follows that the right vertical map in \eqref{equation:2completedsquare1} is surjective on homotopy groups.
Indeed, 
recall that $\SH^{fin}$ is the initial stable symmetric monoidal $\infty$-category according to \cite[Theorem 3.1]{BGT}.
Thus for any symmetric monoidal stable $\infty$-category $\scr C$ and symmetric monoidal functor $F: \scr C \to \SH(\C)_2^\comp$, 
there exists a factorization 
\[
(\pi_*^s)_2^\comp \to \pi_*(c_2^\comp) \xrightarrow{F} \pi_{*,0}((\1_\C)_2^\comp)
\]
and the composite is surjective by Levine's result. 
Thus using \cite[Corollary 6.43]{A1-alg-top} we deduce the pullback square of rings
\begin{equation}
\label{equation:2completedsquare2}
\begin{CD}
\pi_{0,0}(\1_2^\comp)({\scr O}_{F}^{\prime}) @>>> \bigoplus^{r}\GW(\R)_2^\comp \oplus \bigoplus^{c} \GW(y_{i})_2^\comp \\
@VVV                               @VVV         \\
\GW(x)_2^\comp @>>> \bigoplus^{r+c} \GW(\C)_2^\comp
\end{CD}
\end{equation}
%\begin{equation}
%\label{equation:2completedsquare2}
%\begin{CD}
%\pi_0(\1_2^\comp)(\Z[1/2]) @>>> \GW(\R)_2^\comp \\
%@VVV                               @VVV         \\
%\GW(\F_3)_2^\comp @>>> \GW(\C)_2^\comp
%\end{CD}
%\end{equation}
Note that \eqref{equation:2completedsquare2} comes from a diagram in $\Premot_{{\scr O}_{F}^{\prime}}$.
Hence the maps in \eqref{equation:2completedsquare2} are $\pi_{0,0}(\1_2^\comp)({\scr O}_{F}^{\prime})$-algebra maps, 
so a fortiori $\Z[({\scr O}_{F}^{\prime})^\times]$-algebra maps.
The Grothendieck-Witt rings in \eqref{equation:2completedsquare2} are quotients of 
$\Z[({\scr O}_{F}^{\prime})^\times]_2^\comp$; 
thus the lower horizontal and right-hand vertical maps in \eqref{equation:2completedsquare2} are unique 
$\Z[({\scr O}_{F}^{\prime})^\times]$-algebra maps.
Thus \eqref{equation:2completedsquare2} is the $2$-adic completion of \eqref{equation:2completedsquare1} and there is an 
isomorphism of $\Z[({\scr O}_{F}^{\prime})^\times]$-algebras
\[ 
\pi_{0,0}(\1_2^\comp)({\scr O}_{F}^{\prime}) \wequi \GW({\scr O}_{F}^{\prime})_2^\comp
\] 

There is a similar pullback square for $\pi_{1,0}(\ph) \otimes \Q$. 
Since the vanishing $\pi_{1,0}(\1_2^\comp)(k) \otimes \Q = 0$ holds for $k = \R$ \cite[Figure 4]{dugger2016low}, 
$k=\C$ \cite[Corollary 2]{levine2014comparison}, 
and $k=\F_q$ \cite[Theorem 1.3]{wilson2016motivic}, 
we deduce the vanishing 
\[
\pi_{1,0}(\1_2^\comp)({\scr O}_{F}^{\prime}) \otimes \Q = 0
\]
Inserted into the fracture square long exact sequence we get a pullback square of 
$\Z[({\scr O}_{F}^{\prime})^\times]$-algebras
\begin{equation}
\label{equation:pullbackAalgebras}
\begin{CD}
\pi_{0,0}(\1)({\scr O}_{F}^{\prime}) \otimes \Z_{(2)} @>>> \pi_{0,0}(\1)({\scr O}_{F}^{\prime}) \otimes \Q \\
@VVV                                          @VVV                     \\
\pi_{0,0}(\1_2^\comp)({\scr O}_{F}^{\prime})          @>>> \pi_{0,0}(\1_2^\comp)({\scr O}_{F}^{\prime}) \otimes \Q \\
@|                                            @|                       \\
\GW({\scr O}_{F}^{\prime})_2^\comp                @.   \GW({\scr O}_{F}^{\prime})_2^\comp \otimes \Q
\end{CD}
\end{equation}
By inspection there are isomorphisms of $\Z[({\scr O}_{F}^{\prime})^\times]$-algebras
\begin{align*} 
\pi_{0,0}(\1)({\scr O}_{F}^{\prime}) \otimes \Q &
%\stackrel{(*)}{\wequi} 
\wequi H^0_\ret({\scr O}_{F}^{\prime},\Q) \times H^0({\scr O}_{F}^{\prime},\Q) \\ 
&\wequi \Q^{r} \times \Q \\ &\wequi \GW({\scr O}_{F}^{\prime}) \otimes \Q
\end{align*} 
We refer to \cite[Theorem 7.2]{bachmann-SHet} for a proof of the first isomorphism.
Since $\pi_{0,0}(\1)({\scr O}_{F}^{\prime}) \otimes \Q$ is a quotient of $\Z[({\scr O}_{F}^{\prime})^\times] \otimes \Q$, 
in \eqref{equation:pullbackAalgebras}, 
the right-hand vertical map 
\[
\pi_{0,0}(\1)({\scr O}_{F}^{\prime}) \otimes \Q \wequi \GW({\scr O}_{F}^{\prime}) \otimes \Q 
\to 
\pi_{0,0}(\1_2^\comp)({\scr O}_{F}^{\prime}) \otimes \Q \wequi \GW({\scr O}_{F}^{\prime})_2^\comp \otimes \Q
\]
is the unique $\Z[({\scr O}_{F}^{\prime})^\times]$-algebra map.
This shows we can identify the square of $\Z[({\scr O}_{F}^{\prime})^\times]$-algebras \eqref{equation:pullbackAalgebras} 
with the corresponding fracture square for $\GW({\scr O}_{F}^{\prime})\otimes \Z_{(2)}$.
It also follows that the unit map to $\KO$ induces an isomorphism, 
since $\pi_{0,0}(\KO_{\scr O_F}) = \GW(\scr O_F')$ is a quotient of $\Z[\scr O_F'^\times]$.

Next we show the vanishing $\pi_{*,0}(\1_{\scr O_F'}) \otimes \Z_{(2)} = 0$ for $*<0$.
From \eqref{equation:2completedsquare1}, 
since $a$ is surjective on $\pi_*$ and all terms except possibly the top left vanish on $\pi_*$ for $*<0$, 
we deduce that $\pi_{*,0}((\1_{\scr O_F'})_2^\comp) = 0$ for $*<0$.
Since $\GW(\scr O_F')$ is finitely generated, 
the map 
\[ 
\GW(\scr O_F') \otimes \Q \times \GW(\scr O_F')_2^\comp \to \GW(\scr O_F')_2^\comp \otimes \Q 
\] 
is surjective.
Considering the fracture square for $\pi_{*,0}(\1_{\scr O_F'}) \otimes \Z_{(2)}$ it thus remains to prove
$\pi_{*,0}(\1_{\scr O_F'}) \otimes \Q = 0$ for $*<0$.
This follows from the identification of these groups with subquotients of the rational gamma filtration and rational 
real étale cohomology, 
both of which vanish in these degrees, as above.
\end{proof}

Applying the same proof method establishes the following odd-primary analog of Theorem \ref{theorem:general2mainthmintro}.

\begin{theorem}  \tom{think about this again before submission}
\label{theorem:generalellmainthmintro}
Let $\ell$ be an odd prime number.
Suppose $F$ is $\ell$-regular and $\mu_{\ell}\subset F$. 
For the endomorphism ring of $\1_{{\scr O}_{F}^{\prime}}$ over the base scheme ${\scr O}_{F}^{\prime}:= \scr O_F[1/\ell]$ 
there is an isomorphism of $\Z[({\scr O}_{F}^{\prime})^\times]$-algebras
\[ 
\pi_{0,0}(\1_{{\scr O}_{F}^{\prime}}) \otimes \Z_{(\ell)} 
\wequi 
\GW({\scr O}_{F}^{\prime}) \otimes \Z_{(\ell)}.
\]
%induced by the unit map $\1 \to \KO$.
Moreover, 
we have the vanishing result
\[ 
\pi_{*,0}(\1_{\scr O_F'}) \otimes \Z_{(\ell)} 
= 
0 
\quad\text{for}\quad *<0
\]
The same results hold for the motivic sphere over the base scheme $\mathbb{Z}[1/\ell]$ when $\ell$ is a regular prime.
\end{theorem}

\bibliographystyle{alpha}
\bibliography{topmod}

\end{document}